\documentclass[10pt]{article}
\usepackage[margin=1in]{geometry}
\usepackage{float}
  \usepackage{amsmath,amssymb,amsthm}
  \usepackage{graphicx,mathtools}
  \usepackage{xcolor,soul}
\usepackage{enumerate}
\usepackage{verbatim}
\usepackage{marginnote}
\usepackage{array}
\usepackage{booktabs}

\newcolumntype{L}[1]{>{\raggedright\let\newline\\\arraybackslash\hspace{0pt}}m{#1}}

\numberwithin{equation}{section}

\newtheorem{theorem}{Theorem}[section]
\newtheorem{cor}[theorem]{Corollary}
\newtheorem{lemma}[theorem]{Lemma}

\newtheorem{prop}[theorem]{Proposition} \theoremstyle{definition}
\newtheorem{definition}[theorem]{Definition}
 \theoremstyle{remark}
\newtheorem{rem}[theorem]{Remark}

\theoremstyle{plain}
\newtheorem{assumption}{Assumption}

\newcommand{\col}[1]{\begin{bmatrix}#1\end{bmatrix}}

\newcommand{\bra}[1]{\left[#1\right]} 
\newcommand{\mat}[1]{\begin{matrix}#1\end{matrix}} 
\newcommand{\bmat}[1]{\bra{\mat{#1}}} 

\def\beq{\begin{equation} } \def\eeq{\end{equation}}

\usepackage{mathtools}

\DeclarePairedDelimiterX\set[2]{\{}{\}}{#1\,\delimsize\vert\,#2}

\setstcolor{red}

\newcommand\comments[1]{\textcolor{orange}{#1}}

\def\RR{\mathbb R}

\def\ben{\begin{enumerate} }

\def\een{\end{enumerate} }

\def \p { \partial}
 \def \r{ \rangle}

 \def \f{ \mathbf{f}}

\def \Id{ \text{Id} } 
\def \div{ \text{div}}  

\def \n{ \nabla}
  
\def \t{ \mathbf{t}}

\def \htau{ \hat{\tau}}

\def \WF{ \text{WF}} 

\def \BtoB{_{\p \to \p}}
\def \CtoS{_{\mathbf C \to \mathbf S}}
\def \CtoB{_{\mathbf C \to \p}}
\def \BtoS{_{\p \to \mathbf S}}
\def \Op{ Q} 
\newcommand \bad{bad}

\newcommand\DT{\ensuremath{_{\text{\normalfont DT}}}}		
\newcommand\MDT{\ensuremath{_{\text{\normalfont MDT}}}}	

\renewcommand{\restriction}{\mathord{\upharpoonright}}
\def \DMDT{ \mathcal D \MDT}

\newcommand\tail{\ensuremath{_{\text{tail}}}}
\newcommand\subin{\ensuremath{_{\text{in}}}}
\newcommand\subout{\ensuremath{_{\text{out}}}}
\newcommand\prin{\ensuremath{^{\text{prin}}}}
\newcommand\PS{{P\!/\!S}}
\newcommand\red{\textcolor{red}}

\makeatletter
\def\blfootnote{\gdef\@thefnmark{}\@footnotetext}
\makeatother

\title{Recovery of discontinuous Lam\'{e} parameters from local dynamic boundary data}

\author{Peter Caday
\thanks{Intel, Hillsboro OR, USA. \texttt{peter.caday@gmail.com}}
\and
Maarten V. de Hoop
\thanks{Simons Chair in Computational and Applied Mathematics and Earth Science, Rice University, Houston TX, USA. \texttt{mdehoop@rice.edu}}
\and
Vitaly Katsnelson
\thanks{College of Arts and Sciences, New York Institute of Technology, New York NY, USA. \texttt{vkatsnel@nyit.edu}. }
\and
Gunther Uhlmann
\thanks{Department of Mathematics, University of Washington, Seattle WA, USA. \texttt{gunther@math.washington.edu}}
\thanks{Institute for Advanced Study, The Hong Kong University of Science and Technology, Clear Water Bay, Hong Kong.}
}

\begin{document}

\maketitle
 \begin{abstract}
 Consider an isotropic elastic medium $\Omega \subset \RR^3$ whose Lam\'{e} parameters are piecewise smooth.
 In the elastic wave initial value inverse problem, we are given the solution operator for the elastic wave equation, but only outside $\Omega$ and only for initial data supported outside $\Omega$. Using the recently introduced scattering control series in the acoustic case, we prove that piecewise smooth Lam\'{e} parameters are uniquely determined by this map under certain geometric conditions. We also show the extent that multiple scattering in the interior may be suppressed and eliminated with access to only this partial solution map, which is akin to the dynamic Dirichlet-to-Neumann map.
 \end{abstract}

\section{Introduction}

The wave inverse problem asks for the unknown coefficient(s), representing wave speeds, of a wave equation inside a domain of interest $\Omega$, given knowledge about the equation's solutions (typically on $\p \Omega$). Traditionally, the coefficients are smooth, and the data is the Dirichlet-to-Neumann (DN) map, or its inverse. The main questions are uniqueness and stability: can the coefficients be recovered from the Dirichlet-to-Neumann map, and is this reconstruction stable relative to perturbations in the data? In the case of a scalar wave equation with smooth coefficients, a number of results by Belishev, Stefanov, Vasy, and Uhlmann \cite{Belishev-multidimenIP,UVLocalRay, SURigidity} have answered the question in the affirmative. For the piecewise smooth case, a novel scattering control method was developed in \cite{CHKUControl} in order to show in \cite{CHKUUniqueness} that uniqueness holds as well for piecewise smooth wave speeds with conormal singularities, under very mild geometric conditions. We term that particular method as \emph{blind scattering control} since it assumes absolutely no knowledge of the wavespeed in the interior region, and uses only measurements exterior to $\Omega$. Our goal is to extend these results to the isotropic elastic system. This presents new difficulties due to the lack of the sharp form of the unique continuation result of Tataru since we have to deal with two different wave speeds. \\

In the elastic setting, or for that matter, any hyperbolic equation with multiple wave speeds, the story is far from complete. Consider the isotropic elastic wave equation in a bounded domain $\Omega$. The wave operator for elastodynamics is given as $\Op=\rho\p^2_t - L$ with
\[
L = \nabla \cdot( \lambda \div \otimes \Id + 2\mu \widehat{\nabla}),
\]
$\rho$ is the density, $\lambda$ and $\mu$ are the Lam\'{e} parameters, and $\widehat{\nabla}$ is the \emph{symmetric gradient} used to define the strain tensor for an elastic system via $\widehat{\nabla} u = (\nabla u + (\nabla u)^T)/2$ for a vector valued function $u$. Operator $\Op$ acts on a vector-valued distribution $u(x,t)= (u_1,u_2,u_3)$, the \emph{displacement} of the elastic object.
For the isotropic, elastic setting with smooth parameters, the uniqueness question was settled by Rachele in \cite{Rach00} and Hansen and Uhlmann \cite{HUPolarization}. First, Rachele proved that one can recover the jet of $\lambda$, $\mu$, and $\rho$ at $\p \Omega$ explicitly. In \cite{Rach00, RachBoundary}, she showed that one can recover the $\textit{p}$ and $\textit{s}$ wave speeds in $\Omega$ provided the hyperbolic DN map is known on the whole boundary and assuming strict geometry that preclude caustics. Hansen and Uhlmann studied the problem with a residual stress, allowing conjugate points and caustics, and showed that one can recover both lens relations and derived the consequences of that. These are all results for the global problem where the DN map is known on the whole boundary. Stefanov, Vasy, and Uhlmann \cite{UVLocalRay,SUVRigidity} have extended these results to the local inverse problem using the Uhlmann-Vasy methods on the local geodesic ray transform \cite{UVLocalRay} and using a pseudolinearization first developed in \cite{SURigidity}. They are able to do a local recovery of both wave speeds that depend on three parameters $\lambda, \mu, \rho$. There are also related inverse problems in \emph{thermoacoustic tomography} where one tries to recover a source (initial condition) rather than a PDE parameter \cite{Tittelfitz-TAT, Kuchment-mathOfTAT,Kuchment-ReconstructionTAT,SU-TATVariable,SU-TATBrain}.\\

No such results are known for when the elastic parameters have interfaces (conormal singularities). Even the blind scattering control method, which is very similar to boundary control and was used to prove uniqueness in \cite{CHKUUniqueness} for the acoustic setting, does not readily apply here. The reason is very simple: although unique continuation results hold for the elastic setting, they are far weaker, being based on the slowest wave speed, and so the boundary control method is not known to work since it is not possible, or at least not known, how to decouple the elasticity system completely even though it is easy to do that microlocally. A Lam\'{e} type of system having the same principal part which can be decoupled fully was studied by Belishev in \cite{BelLameType} and the boundary control method (see \cite{Belishev-BCInRecontr,Belishev-dynamicalVariant,Belishev-multidimenIP}) was used for unique recovery. Such an approach only worked because the system was able to fully decouple so that the scalar boundary control methods would apply to the decoupled constituents. Therefore, it fails for the piecewise smooth setting where the coupling of different modes at the interfaces is unavoidable, and so it does not simplify matters here to study Belishev's Lam\'{e} type system with piecewise smooth parameters. Instead, we focus on a geometric uniqueness problem analogous to \cite{SUVRigidity} and employ a layer stripping argument to utilize the results in \cite{UVLocalRay,SUVRigidity} in the smooth case.
\\

 The main result of this paper is that under certain geometric assumptions, we show unique determination of Lam\'{e} parameters that contain singularities via microlocal analysis, scattering control, and a layer stripping argument akin to \cite{SUVRigidity}.
 Most proofs are microlocal to avoid using unique continuation results, but we require an important geometric assumption, which is a \emph{convex foliation condition} (see \S\ref{s: convex foliation}) for each wave speed $c_\PS$. As mentioned in \cite{SUVRigidity}, for a particular wave speed, this condition relates to the existence of a function with strictly convex level sets, which in particular holds for simply connected compact manifolds with strictly convex boundaries such that the geodesic flow has no focal points (lengths of non-trivial Jacobi fields vanishing at a point do not have critical points), in particular if the curvature of the manifold is negative (or just non-positive). Also, as explained in \cite{SUV-ElasticLocalRay}, if $\Omega$ is a ball and the speeds increase when the distance to the center decreases (typical
for geophysical applications), the foliation condition is satisfied.\\

 The other key ingredient is that even though a lens map does not make sense with internal multiples present, one may use a scattering control-like process introduced in \cite{CHKUControl} (not blind scattering control which requires sharp unique continuation theorems) to recover lens data for singly reflected rays. This construction will also be entirely microlocal and circumvents the need for unique continuation results. We denote by $u_h$ the solution to the homogeneous elastic equation on $\RR^3$ with initial time Cauchy data $h$. All of our function spaces are of the form $X(\cdot; \mathbb{C}^3)$ since we have vector valued functions in the elastic setting, but throughout the paper, we will not write the vector valued part $\mathbb{C}^3$ to make the notation less burdensome. Let $\overline{\Omega}^c$ be the complement of $\overline{\Omega}$ and we define the \emph{outside measurement operator} $\mathcal F: H^1_c(\overline{\Omega}^c) \oplus L^2_c(\overline{\Omega}^c) \to C^0(\RR_t; H^1(\overline{\Omega}^c)) \cap  C^1(\RR_t; L^2(\overline{\Omega}^c))$ as
\[
\mathcal F: h_0 \to u_{h_0}(t)|_{\overline{\Omega}^c}.
\]
 Due to a technicality, we use slightly different sets for our measurement region than $\overline{\Omega}^c$ in the main body, but the idea is the same. The operator $\mathcal F$ only measures waves outside $\Omega$ after undergoing scattering within $\Omega$, and it is associated to a particular elastic operator $\Op$ with a set of parameters. Given a second set of elastic parameters $\tilde\lambda$, $\tilde\mu$ we obtain analogous operators $\tilde \Op$ and $\tilde{\mathcal F}$. Denote the associated $\PS$ wave speeds $c_{\PS}$ and $\tilde c_{\PS}$. From here on, we use $\PS$ to refer to either subscript or wave speed. In addition, to avoid the technical difficulties of dealing with corners or higher codimension singularities of $c_{\PS}$, we always assume that the singular support of $c_{\PS}, \tilde c_{\PS}$ lies in a closed, not necessarily connected hypersurface in $\Omega$; we will deal with corners and edges in a separate paper. We will prove the following result.

\begin{theorem}
Assume $\mathcal F = \tilde {\mathcal F}$, and that $c_{\PS},\tilde c_{\PS}$ satisfy a certain geometric foliation condition. Then $c_P = \tilde c_P$ and $c_S = \tilde c_S$ inside $\Omega$.
\end{theorem}

 Via a layer stripping approach, we will obtain local travel time data and lens relations at the current layer from $\mathcal F$ \footnote{The fact that the interfaces are not dense makes this possible theoretically in the sense that there will exist an open set of rays at the current layer that \emph{do not cross any interfaces} after a finite time when they are close to being tangent to the layer.}. To do this, we will employ an analogue of the microlocal scattering control construction appearing in \cite[section 5]{CHKUControl} to create specific $P$ or $S$ waves at the current, deepest layer to extract local travel time data and lens relations without having the internal multiples interfere with recovery of this data. Without such techniques, one would not be able to distinguish waves that contain this subsurface travel time data from internal multiples created from the conormal singularities of the Lam\'{e} parameters.\\

Since we use a scattering control parametrix in the above proof coupled with knowledge of the Lam\'{e} parameters on part of the interior, it becomes natural to address an old question, albeit quite different than the one posed above, on how much multiple scattering may one control and eliminate in the elastic setting having access \emph{only} to the \emph{outside measurement operator} and no knowledge of the Lam\'{e} parameters in the interior. That is, how much of blind scattering control addressed in the acoustic setting in \cite{CHKUControl} applies in the elastic setting? A partial answer for the elastic setting is in \cite{Wapenaar-Elastic}, where they assume knowledge of first arrival times corresponding to purely transmitted $P$ waves, $S$ waves, and certain mode converted waves. This is a strong assumption since a single wave packet entering $\Omega$ produces numerous scattered waves that one measures at the surface and one cannot \emph{a priori} associate travel times with a particular primary reflected wave versus a secondary reflected wave. We wish to dispense with such an assumption and in Appendix \ref{s: blind scatt control}, we show how with blind scattering control, one may control all the internal multiples within the $S$-domain of influence of the initial source, and certain additional multiples beyond this region. Essentially, any multiples that may be recreated with a tail behind the initial source may also be controlled and eliminated. It should be emphasized that this result is quite disparate from the main result of the paper since we do a time reversal procedure of \emph{all} the scattered data one obtains outside $\Omega$, while in the main theorem, one does a parametrix construction using very \emph{specific}, reflected $\PS$-waves to eliminate particular waves in the interior.
\\

\section{The data map}
We will use this section to give the basic definitions and setup for the main theorem.

\subsection{Geometric setup}
Let $\Omega$ be a bounded region in $\RR^3$ with smooth boundary. It represents a linearly elastic, inhomogeneous, isotropic object. We assume the Lam\'{e} parameters $\lambda(x)$ and $\mu(x)$ satisfy the \emph{strong convexity condition}, namely that $\mu>0$ and $3\lambda+2\mu >0$ on $\overline{\Omega}$. Also, assume the parameters $\lambda, \mu$ lie inside $L^{\infty}(\Omega)$ and that $\lambda, \mu$ are piecewise smooth functions that are singular only on a set of disjoint, closed, connected, smooth hypersurfaces $\Gamma_i$ of $\overline{\Omega}$, called \emph{interfaces}. We also set $\Gamma = \bigcup \Gamma_i$ to be the collection of all the interfaces.

 The two wave speeds are $c_P = \sqrt{(\lambda + 2\mu)/\rho}$ and $c_S = \sqrt{\mu/\rho}$, where $\rho$ is the density. In particular, this ensures that $c_P > c_S$ on $\overline{\Omega}$.
As in \cite{CHKUControl}, we will probe $\Omega$ with Cauchy data (an \emph{initial pulse}) concentrated close to $\Omega$ with a particular polarization, in some Lipschitz domain $\Theta \supset \Omega$. While we are interested in what occurs inside $\Omega$, the initial pulse region $\Theta$ will actually play a larger role in the analysis.

  Since we take measurements outside $\Omega$, let us extend the Lam\'{e} parameters to all of $\RR^n$ so that they are smooth outside $\overline{\Omega}$ and our wavefields are now well-defined there as well. We will denote by
\[
g_{\PS} = c_{\PS}^{-2}dx^2
\]
the two different metrics associated to the rays. As in \cite{CHKUControl}, we can define the distance functions $d_{\PS}(\cdot, \cdot)$ corresponding to the respective metrics by taking the infimum over all lengths of the piecewise smooth paths between a pair of points. Here and throughout the paper, a $\PS$ subscript indicates either $P$ or $S$ subscripts.

Now, define the \emph{$P$-depth} $d^*_{\Theta}(x)$ of a point $x$ inside $\Theta$:
\[
d^*_{\Theta}(x) = \begin{cases}
+d_P(x,\p \Theta), & x \in \Theta ,\\
-d_P(x, \p \Theta), & x \notin \Theta.
\end{cases}.
\]

We use the (rough) metric $g_P$ since finite speed of propagation for elastic waves is based on the faster $P$-wavespeed. We will add to the initial pulse a Cauchy data control (a \emph{tail}) supported outside $\Theta$, whose role is to remove multiple reflections up to a certain depth, controlled by a time parameter $T \in (0,\frac{1}{2}\text{diam}_P\Omega)$. This will require us to consider controls supported in a sufficiently large Lipschitz neighborhood $\Upsilon\subsetneq\RR^3$ of $\overline{\Theta}$ that satisfies $d_S(\partial \Upsilon, \overline{\Theta})>2T$ and is otherwise arbitrary. It will be useful to define $\Theta^\star=\set*{x \in \Upsilon}{d^*_{\Theta}(x)<0}$.

\subsection{Elastic waves}\label{s: basic setup}
 Recall the wave operator for elastodynamics $\Op$ discussed in the introduction given as $\Op=\rho \p^2_t - L$ with
\[
L = \nabla \cdot( \lambda \div \otimes \Id + 2\mu \widehat{\nabla}).
\] Let us also recall the characteristic set $\Op$ defined in \cite{Rach00} and \cite{HUPolarization}. It consists of two mutually disjoint sets $\Sigma_P,\Sigma_S \subset T^*\RR^3$ where $\Sigma_{\PS}$ are the characteristic sets for the scalar wave operators $c_{\PS}^{-2}\p^2_t - \Delta$.

Let $\tilde{\mathbf{C}}$ be the space of Cauchy data of interest:
\[
\tilde{\mathbf{C}} = H^1_0(\Upsilon;\mathbb{C}^3) \oplus L^2(\Upsilon; \mathbb{C}^3)
\]
although we will suppress the `$\mathbb{C}^3$' notation when it is clear from the context. We equip the space with the \emph{elastic energy inner product}
\[
\langle (f_0,f_1),(g_0,g_1) \rangle = \int_{\Omega} \left(f_1 \cdot \bar{g}_1  +
\lambda(x) \div(f_0) \div(\bar{g}_0)+ 2\mu(x)\widehat{\nabla}f_0 :\widehat{\nabla} \bar{g}_0\right)\,dx.
\]
Within $\tilde {\mathbf C}$, define the subspaces of Cauchy data supported inside and outside $\Theta_t$:
\begin{align*}
\mathbf H &= H^1_0(\Theta) \oplus L^2(\Theta), \qquad \qquad
\tilde{\mathbf H}^\star = H^1_0(\Theta^\star) \oplus L^2(\Theta^\star).
\end{align*}
Define the energy of Cauchy data $h=(h_0,h_1) \in \tilde{\mathbf C}$ in a subset
$W \subset \RR^3$:
\[
\mathbf E_W(h):= \int_W \left(
\lambda(x) |\div(h_0)|^2 + \mu(x)|\widehat{\nabla}h_0|^2
+ |h_1|^2
\right)dx.
\]
We need another definition, related to the harmonic extensions used in \cite{CHKUControl}.
\begin{definition}
We say that $(f,g) \in H^1(\RR^3) \oplus L^2(\RR^3)$ is \emph{stationary} on a domain $U$ if $Lf = g =0$ on $U$.
\end{definition}

Next, define $F$ to be the solution operator for the elastic wave initial value problem:
\begin{equation}
F: H^1(\RR^n)\oplus L^2(\RR^3) \to C(\RR, H^1(\RR^3))
\qquad
F(h_0,h_1)=u \text{ s.t. }
\begin{cases}
	\Op u &= 0,\\
				u\restriction_{t=0} &= h_0,\\
			\p_t u\restriction_{t=0} &= h_1.
	\end{cases}
\end{equation}
Let $R_s$ propagate Cauchy data at time $t=0$ to Cauchy data at $t=s$:
\begin{equation}
R_s = (F, \p_t F)\Big|_{t=s}:H^1(\RR^3) \oplus L^2(\RR^3) \to H^1(\RR^3)\oplus L^2(\RR^3).
\end{equation}
Now combine $R_s$, with a time-reversal operator $\nu: \tilde{\mathbf C}
\to \tilde{ \mathbf C}$, defining for a given $T$
\[
R = \nu \circ R_{2T}, \qquad \qquad \nu: (f_0,f_1) \mapsto (f_0,-f_1).
\]
In our problem, only waves interacting with $(\Omega,\mu,\lambda)$ in the time interval $[0,2T]$ are of interest. Consequently, let us ignore Cauchy data not interacting with $\Theta$, as follows.

Let $\mathbf G=\tilde{\mathbf H}^\star\cap\big( R_{2T}(H^1_0(\RR^3\setminus\overline\Theta)\oplus L^2(\RR^3\setminus\overline\Theta))\big)$ be the space of Cauchy data in $\tilde{\mathbf C}$ whose wave fields vanish on $\Theta$ at $t=0$ and $t=2T$. 
Let $\mathbf C$ be its orthogonal complement inside $\tilde{\mathbf C}$, and ${\mathbf H}^\star$ its orthogonal complement inside $\tilde{\mathbf H}^\star$. With this definition, $R_{2T}$ maps $\mathbf C$ to itself isometrically.

In Appendix \ref{s: blind scatt control} we show how much scattering one may control in the elastic setting without any knowledge of the parameters inside $\Omega$ based on unique continuation theorems that are much weaker due to the presence of multiple wave speeds. Despite not having full unique continuation, one is still able to construct parametrices for the elastic equation that are approximate solutions to the elastic equation. With such parametrices, we will use the principles from scattering control to obtain ``subsurface'' travel time data by eliminating certain scattered constituents microlocally. This is precisely the task we pursue in the remainder of the paper. First, we develop our main technical tool, which is the parametrix for elastic wave solutions.

\section{Unique recovery of piecewise smooth \textit{P}- and \textit{S}-wave speeds: A geometric point of view}

We consider the problem of showing that both the $P$ and $S$ wavespeeds are determined by the outside-measurement-operator under certain geometric conditions that give us access to all the requisite rays.

\subsection{Foliation condition}\label{s: convex foliation}

 First, since our proof of the main theorem will require recovery of all the parameters in a layer stripping argument, we make a simplifying assumption and assume the density \[\rho =1\] throughout the paper. We will explain in section \ref{s: discussion} how one might dispense with such an assumption.

Let us recall all the definitions from \cite{CHKUUniqueness}, adapted to the elastic setting. We start by extending the convex foliation condition to our piecewise smooth setting, keeping in mind that $\Gamma_i,\Gamma$ are the interfaces defined in section \ref{s: basic setup}.
\begin{definition}
$\rho: \overline \Omega \to [0,\tau_0]$ is a (piecewise) \emph{convex foliation} for $(\Omega,c_{\PS})$ (meaning for both $c_P$ and $c_S$ simultaneously) if the following conditions hold:
\begin{itemize}
\item
$\p \Omega = \rho^{-1}(0)$ and $\rho^{-1}(\tau_0)$ has measure zero;

\item
$\rho$ is smooth and $d\rho \neq 0$ on $\rho^{-1}((0,\tau_0))\setminus \Gamma$.

\item each level set $\rho^{-1}(t)$ is geodesically convex with respect to $c_p$ and $c_s$ when viewed from $\rho^{-1}((t,T))$, for $t \in [0,\tau_0)$;

\item the interfaces of $c_{\PS}$ are level sets of $\rho_i$, that is $\Gamma_i \subset \rho^{-1}(t_i)$ for some $t_i$;
\item $\rho$ is upper semicontinuous.
\end{itemize}
We say that $(c_P,c_S)$ satisfies the \emph{convex foliation condition} if there exists a convex foliation for $(\Omega, c_\PS)$.
\end{definition}

Having interfaces being part of the foliation allows for some unusual configurations. In addition, the leaves of the foliation may have intricate, non-trivial topologies and the geometry can be complicated as well, allowing conjugate points (see below Figure \ref{f: foliation pic}).
\begin{figure}[H]
\centering
             \includegraphics[scale=0.8]{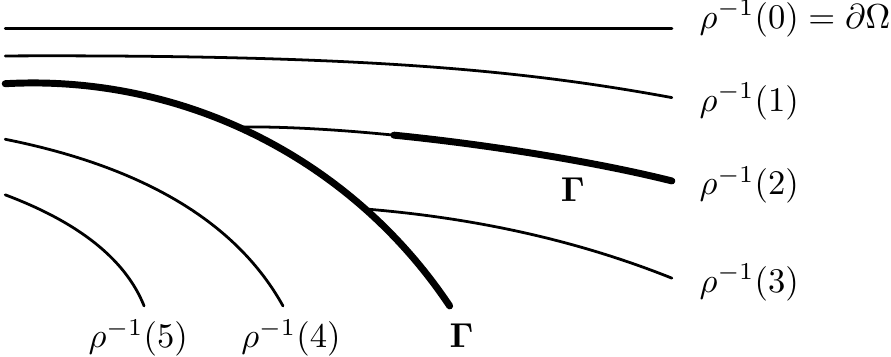}
\caption{An example of a piecewise convex foliation. Thick lines indicate the interfaces $\Gamma$; thin lines trace selected level sets of the foliation function $\rho$, which is allowed (but not required) to be singular at $\Gamma$.}
\label{f: foliation pic}
\end{figure}

From now on we assume
\begin{assumption}
$(\Omega, c_\PS)$ satisfy the convex foliation condition.
\end{assumption}

 Our main theorem is

\begin{theorem}\label{thm: main theorem}
Under the convex foliation condition, if $\mathcal F = \tilde {\mathcal F}$, then $c_{\PS} = \tilde c_{\PS}$.
\end{theorem}

We note that the case where $c_P$ and $c_S$ have separate foliations does not add much more generality to the theorem (see Remark \ref{rem: weakening foliation}).

\begin{rem}\label{rem: weakening foliation}
We are assuming that the level sets of one function $\rho$ produces a convex foliation for both the $c_P$ and $c_S$ wave speed. One may wonder whether this is strictly necessary since we recover the wave speeds one at a time in the proof. Upon close examination of the main proof, it will be vital that the interfaces coincide with the leaves of the foliation so that we get the correct scattering behavior that ensures enough branches of a particular ray return outside $\Omega$, which is the measurement region. Thus, we may allow $c_P$ and $c_S$ to have different foliations, but the foliations must still coincide at and near each interface. Dealing with interfaces is the main novelty of the paper so allowing different foliations away from the interfaces does not present much novelty to our results.
\end{rem}

We note that convex foliation gives us crucial information on the reflected and transmitted waves emitted when an incident wave hits an interface. Indeed, define
\begin{align}
	&\begin{aligned}
		\Omega_\tau&=\rho^{-1}\big((\tau,\tau_0]\big),\\
		\Omega^\star_\tau &= \rho^{-1}\big((0,\tau)\big),
	\end{aligned}	
	&
	\Sigma_\tau&= \rho^{-1}(\tau).
\end{align}

Also, let $\Sigma_\tau^\pm$ denote the two sides of the interface, where $(-)$ refers to the outside of $\Omega_\tau$ (facing decreasing $\tau$) and $(+)$ the inside. We also fix such notation for the remainder of the paper. We have the corresponding sets of $\PS$ hyperbolic points $\mathcal H_{\PS}^\pm \subset T^*\Gamma_\pm$ (see \cite[section 4]{HUPolarization} for the relevant definitions). The convexity guarantees that $\mathcal H^+_P \subset \mathcal H_P^-$ with an analogous statement for the $S$ hyperbolic set. Thus, a $P$ wave hitting $\Sigma_\tau$ from below must produce a transmitted $P$ wave. In fact, it must produce a transmitted $S$ wave as well since $c_P > c_S$. The same holds for an $S$ wave hitting $\Gamma$ from below, but a mode conversion in the transmitted wave does not necessarily occur since mode conversions only occurs up to a critical angle. Thus, there is no total internal reflection from below the interfaces.

First, we need several definitions taken from \cite{CHKUUniqueness} extended to the elastic setting.

\begin{definition}\label{def: foliation downward}
A \emph{foliation downward} (resp. \emph{upward}) covector $(x,\xi)$ is one pointing in direction of increasing (resp. decreasing) $\rho$. Define $T^*_{\pm}\Omega$ to be the associated open sets:
\[
T^*_{\pm}\Omega = \set*{(x,\xi) \in T^*\Omega}{\pm \langle \xi, d\rho \rangle > 0}.
\]
Hence, we can speak of covectors $(x,\xi)$ pointing \emph{upward/downward} with respect to the foliation.
\end{definition}

\begin{definition}
A (unit-speed) \emph{broken geodesic} in $(\overline\Omega,c_{\PS})$ is a continuous, piecewise smooth path $\gamma: \RR \supset I \to \overline\Omega$ such that each smooth piece is a unit-speed geodesic with respect to either $g_P$ or $g_S$ on $\overline\Omega \setminus \Gamma$, intersecting the interfaces $\Gamma$ at discrete set of points $t_i \in I$. Furthermore, at each $t_i$ the intersection is transversal and Snell's law for reflection and refraction of elastic waves is satisfied. A \emph{broken bicharacteristic} is a path in $T^*\Omega$ of the form $(\gamma,\gamma'^\flat)$, the flat operation taken with respect to $g_P$ or $g_S$ as appropriate. Note that a broken geodesic defined this way may contain both $P$ and $S$ geodesic segments.

A \emph{transmitted broken geodesic} (a concatenation of smooth $P$ and $S$ geodesics) is a unit-speed broken geodesic experiencing only refractions; that is, the inner products of $\gamma'(t_i^-)$ and $\gamma'(t_i^+)$ with the normal to $\Gamma$ have identical signs at each $t_i$. A \emph{transmitted broken bicharacteristic} is then defined analogously.
\end{definition}

\begin{definition}
Let $(x,\xi) \in T^*_+ \overline \Omega \setminus 0$, and $\tau = \rho(x)$. If there exists a purely transmitted bicharacteristic $\gamma$ (with either only $P$ or only $S$ branches) and $\lim_{t \to 0^+}\gamma(t) = (x,\xi)$, we define the \emph{subsurface travel time} $l_{\PS,\tau}(x,\xi)$ as the unique $l>0$ for which $\gamma(l) \in T_-^*\Omega \cap T^* \Sigma_\tau$, and the (\emph{subsurface}) lens relation $L_{\PS,\tau}(x,\xi) = \gamma(l)$.

If $\mathcal{D}_{\PS}$ is the set of $(x,\xi)$ for which such $\gamma$ exists, extend $L_{\PS}$ to $(\overline{\mathcal D_{\PS}}\setminus 0) \setminus T^*\Omega |_{\Gamma}$ by continuity. On the interfaces $T^*\Omega|_\Gamma$, define $L_{\PS}$ by continuity from below.
\end{definition}

\begin{definition}
Let $\Omega_r\subseteq\Omega$ be the set of \emph{regular points}, where $x$ is regular if it is regular with respect to both $c_P$ and $c_S$, as defined in \cite[Definition 3.2]{CHKUUniqueness}.
\end{definition}
Essentially, $x \in \Omega_r$ means that there is a purely transmitted broken $P$ and $S$ geodesic that starts normal to $\p \Omega$ and passes through $x$.
We do not go into detail on the definition of $\Omega_r$ since due to the convex foliation assumption, it is a dense set in $\Omega$, which is all that we use in our proofs:
\begin{lemma}
If $\,\Omega$ is compact, then $\Omega_r$ is dense in $\Omega$ under the convex foliation assumption.
\end{lemma}
The proof is by applying Lemma 3.3 in~\cite{CHKUUniqueness} to both $P$ and $S$ speeds.
\\

Since the proof of the main theorem is microlocal, we must first construct a parametrix for the elastic operator when the Lam\'{e} parameters are piecewise smooth.
\section{Elastic-wave parametrix with scattering}
In this section, we construct the elastic wave parametrix in the presence of singularities in the Lam\'{e} parameters. Most constructions are taken directly from \cite{CHKUControl} used in the acoustic setting and \cite{RachBoundary} in the elastic setting. Other references to construct such parametrices are \cite{Hansen-CPDEInverse} in an acoustic setting and \cite{Dencker-Polarization} for a general systems setting.

 Let us first recall the interfaces $\Gamma_i$, with $\Gamma = \bigcup \Gamma_i$. These hypersurfaces separate $\RR^3\setminus \Gamma$ into disjoint components $\{\Omega_j\}$. We assume each smooth piece of $\lambda$ and $\mu$ extends smoothly to $\RR^3$. In order to distinguish the sides of each hypersurface $\Gamma_i$, consider an \emph{exploded space} $Z$ in which the connected components of $\RR^3 \setminus \Gamma$ are separate. It may be defined in terms of its closure, as a disjoint union
\[
\bar Z = \bigsqcup_j \overline \Omega_j, \qquad \qquad Z= \bigcup_j \Omega_j \subset \overline Z.
\]
In this way, $\p Z$ contains two copies of each $\Gamma_i$, one for each adjoining $\Omega_j$.

When restricting to a particular $\Omega_j$, we may do a microlocal decomposition into the forward and backward propagators as in the acoustic case \cite{Rach00,RachBoundary}. This is because away from the interfaces, $-L$ is a positive elliptic operator with a pseudodifferential square root. See \cite{StolkThesis} for a microlocal construction of this square root. Hence, the construction in \cite[appendix A]{CHKUControl} applies, so for Cauchy data $(f_0,f_1)$ (time $t=0$ say), the Cauchy to solution map may then be decomposed as
\[
F(f_0,f_1) \equiv F^+g_++F^-g_-, \qquad \qquad \col{g_+\\ h_-}
\equiv C \col{f_0\\f_1}
\]
where $C$ is a microlocally invertible matrix $\Psi$DO. The Cauchy data $(g_+,g_-)$ may be interpreted as a single distribution $g$ on a doubled space $\mathbf Z = Z_+ \sqcup Z_-$. The corresponding layers are then $\Omega_{\pm,j}$.

Combining the elastic parametrix construction in Rachele \cite{Rach00} with the scalar wave parametrix in the presence of singularities in the sound speed \cite{CHKUControl}, we may construct a parametrix for $R_T$ in regions where no glancing occurs at an interface. We will describe it as a sum of graph FIOs on $\mathbf Z$ from sequences of reflections, transmissions, and $\PS$ mode conversions, along with operators propagating data from one boundary to another, or propagating the initial data to boundary data.\\

\subsection{Cauchy propagators} To begin, extend each restriction $\mu_j = \smash{\mu\big|_{\Omega_j}}$, $\lambda_j = \smash{\lambda\big|_{\Omega_j}}$ to a smooth function on $\RR^3$. Each $\eta \in T^*\Omega_{\pm,j}$ is associated with a unique $\PS$-bicharacteristic $\gamma^{\PS}_\eta(t)$ in $T^*\RR^3$ passing through $\eta$ at $t=0$, which may escape and possibly re-enter $\Omega_{\pm,j}$, as $t \to \pm \infty$.

To prevent re-entry of wavefronts, we introduce a pseudodifferential cutoff for $\PS$ rays, $\phi^{\PS}(t,x,\xi)$, omitting some details for brevity. Let $t^{\PS}_{e\pm},t^{\PS}_{r\pm}$ denote the first positive and negative escape and reentry times for the $\PS$-ray. We let $\phi^{\PS}(t,\gamma^{\PS}_\eta(t))$ be identically one on $[t^{\PS}_{e-},t^{\PS}_{e+}]$ and supported in $(t^{\PS}_{r-},t^{\PS}_{r+})$. One then modifies $\phi^{\PS}$ on a small neighborhood of $\RR \times T^*\p \Omega_{\pm,j}$ (the glancing $\PS$ rays) to ensure it is smooth.

We then recall the construction of the Cauchy propagators $E_j^{\pm}$ (with $\pm$ corresponding to ``forward'' and ``backward'' propagators) described in detail in \cite{Rach00}:
\begin{equation}\label{e: Cauchy propagator}
(E_j^{\pm} h)_l = \sum_{\PS}\sum_{m=1}^3 (2\pi)^{-3}\left[
\int e^{i\varphi^{\pm}_{\PS}}e^{lm}_{\PS,\pm}\hat h_{m,\pm}(\xi) d\xi
\right]
\end{equation}
with phase functions $\varphi_{\PS}(t,x,\xi)$ and vector-valued amplitudes $e^{lm}_{\PS,\pm}(t,x,\xi)$.

Finally, let $J\CtoS$ be the restriction of $\phi \circ E^{\pm}_j$
 defined by
 \[
 \phi \circ E^{\pm}_j
 :=\sum_{\PS}\sum_{m=1}^3 (2\pi)^{-3}
 \phi^{\PS}(t,x,D_x)\left[
\int e^{i\varphi^{\pm}_{\PS}}e^{lm}_{\PS,\pm}\hat h_{m,\pm}(\xi) d\xi
\right]
 \]
 to $\RR \times \Omega_{\pm,j}$; this is the desired reflectionless propagator.

We also require a variant, denoted $J_{\mathbf C \to \mathbf S+}$, of $J\CtoS$ in which waves travel only forward in time. For this, replace $\phi^{\PS}$ with some $\phi^{\PS}_+$ supported in $(t^{\PS}_{e-},t^{\PS}_{r+})$ and equal to $1$ on $[0,t_{e+}]$. Restricting $J_{\mathbf C \to \mathbf S+}$ to the boundary, we obtain the \emph{Cauchy-to-boundary} map $J\CtoB = J_{\mathbf C \to \mathbf S+}|_{\RR \times \p \mathbf Z}$. One may also construct the boundary-to-solution map, denoted $J\BtoS$, analogous to the above using the construction in \cite{RachBoundary} for the smooth Lam\'{e} parameter case.

 As in \cite[Appendix]{CHKUControl}, $J\BtoS,J_{\mathbf{C}\to \mathbf{S}+} \in I^{-1/4}(\mathbf Z \to \RR \times \mathbf Z)$, and $J\CtoB \in I^0(\mathbf Z \to \RR \times \p \mathbf Z)$. Also, $J\BtoS,J_{\mathbf C\to \mathbf{S}+}$ are parametrices for the elastic equation when applied to $u$ such that $\text{WF}(u)$ lies in an open set $\mathcal V \subset T^*\mathbf Z$ whose $\PS$-bicharacteristics are sufficiently far from glancing. The near-glancing covector set, denoted $\mathcal W$, is $T^*\mathbf Z \setminus \mathcal V$. \\

\subsection{$\PS$-Mode projectors} Since we are in the elastic setting, it will be useful to define microlocal projections $\Pi_{\PS}$ that microlocally project an elastic wavefield $u$ to the respective $P$ and $S$ characteristic sets. Locally and for small times, from (\ref{e: Cauchy propagator}), $u$ has a representation
\[
u = \sum_{\PS}\sum_{m=1}^3 (2\pi)^{-3}\left[
\int e^{i\varphi_{\PS}}e^{lm}_{\PS}\hat h_m(\xi) d\xi
\right]
\]
and so we define
\[
\Pi_P u = \sum_{m=1}^3(2\pi)^{-3}\int e^{i\varphi_{P}}e^{lm}_{P}\hat h_m(\xi) d\xi.
\]
$\Pi_S$ is defined analogously. These definitions can be made global, although it is technically not necessary in our case since our analysis is done near the characteristic set of the elastic operator.
\\

\subsection{Boundary propagators} Outgoing solutions from boundary data $f \in \mathcal D'(\RR \times \mathbf Z)$ may be obtained by microlocally converting boundary data to Cauchy data, then applying $J\CtoS$ as explained in \cite{CHKUControl}. We give a cursory overview of the construction, which translates easily to the elastic setting. The boundary-to-Cauchy conversion can be achieved by applying a microlocal inverse of $J\CtoB$, conjugated by the time-reflecting map $S_s\colon t \to s-t$ for an appropriate $s$. Let $x=(x',x_3)$ be boundary normal coordinates near $\p \Omega_{\pm,j}$. Near any covector $\beta = (t,x'; \tau, \xi') \in T^*\p\Omega_{\pm,j}$ in the hyperbolic region $|\tau| > c^{\PS}_j|\xi'|$, there exists a unique $\PS$-bicharacteristic $\gamma$ passing through $\beta$ and lying inside $\Omega_{\pm,j}$ in some time interval $[s,t), s<t.$\footnote{That is, its projection to $T^*\p \Omega_{\pm,j}$ when it hits $\p T^* \Omega_{\pm,j}$ is $\beta$, but we abuse notation.} Then $J\BtoS$ may be defined by $S_sJ\CtoS J^{-1}\CtoB S_s$ microlocally near $\beta$. The inverse can be seen to exist microlocally away from glancing by ``diagonalizing'' the Cauchy propagators as done in \cite{StolkThesis} and applying the same construction of the scalar wave setting in \cite{CHKUControl}.

On the elliptic region $|\tau|<c_j^{\PS}|\xi'|$ define $J\BtoS$ as a parametrix for the elliptic boundary problem. This may be constructed even in the systems setting as shown in \cite{Taylor75}. Applying a microlocal partition of unity, we obtain a global definition of $J \BtoS$ away from a neighborhood of both $\PS$ glancing regions $|\tau|=c^{\PS}_j|\xi'|$. It can be proven that $J\BtoS \in I^{-1/4}(\RR \times \p \mathbf Z \to \RR \times \mathbf Z)$. Its restriction to the boundary $r_\p \circ J\BtoS$ consists of a pseudodifferential operator equal to the identity on $\mathcal W$ and an elliptic graph FIO $J\BtoB \in I^0(\RR \times \p \mathbf Z \to \RR \times \p \mathbf Z)$ describing waves traveling from one boundary to another.
\\

\subsection{Reflection and transmission} It is well known that trasmitted and reflected waves arise from requiring a weak solution and its normal traction to be $C^0$ at interfaces. Given incoming boundary data $f \in \mathcal E'(\RR \times \p \mathbf Z; \mathbb{C}^3)$ (an image of $J\CtoB$ or $J\BtoB$) microsupported near $\beta$, we seek data $f_R,f_T$ satisfying the interface constraints
\begin{equation*}
f + f_R \equiv \iota f_T,
\end{equation*}
\vskip-3em
\begin{multline*}
\mathcal (\lambda\subin\div(vJ\BtoS vf + J\BtoS f_R))\text{Id}
+2\mu\subin \nabla_s(vJ\BtoS vf + J\BtoS f_R)) \cdot \eta|_{\RR \times \p \mathbf Z}\\
 \equiv \iota (\lambda\subout\div( J\BtoS f_T) Id +2\mu\subout\nabla_s J\BtoS f_T)\cdot \eta|_{\RR \times \p \mathbf Z}
\end{multline*}
Here, $v$ is time-reversal, so $v J\BtoS v$ is the outgoing solution that generated $f$. The map $\iota: \p \mathbf Z \to \p \mathbf Z$ reverses the copies of each boundary component within $\p \mathbf Z$, and $\eta$ denotes the unit normal vector to the interface in question. The subscripts \textit{in} and \textit{out} merely denote which side of the interface one is considering in the Lam\'{e} parameters.

The second equation above simplifies to a pseudodifferential equation
\[
N_1 f + N_R f_R \equiv N_T f_T
\]
with operators $N_1,N_R, N_T \in \Psi^1(\RR \times \p Z; \mathbb{C}^3)$ that may be explicitly computed. The system may be microlocally inverted to recover $f_R = M_R f, f_T = M_T f$ in terms of pseudodifferential reflection and transmission operators $M_R, \iota M_T \in \Psi^0(\RR \times \p \mathbf Z; \mathbb{C}^3)$. In order to compute these operators,
we will use the traction formulation of the interface conditions, which will allow us to use symplectic properties in order to compute and study $M_R$ and $M_T$.

\subsubsection*{Calculating reflection and transmission PsiDO's}

Here, we present the traction representation of the elastic PDE in order to simplify many of the computations, by using the symplectic properties to find inverses. A more general detailed construction may be found in \cite[section 3]{StolkThesis}.

Let us do our analysis locally near an interface $\Gamma$ with boundary normal coordinates such that $\Gamma$ is given by $x_3 = 0$. If we use the traction formulation, and the unit normal $\nu = [0\;\; \;0\;\;\; 1]^\top$, the traction components are
$$ \t_j = \mathbb{C}\nabla_s u \cdot e_j,$$
where $e_j$ are the Euclidean basis vectors, $\mathbb{C}$ is the elastic tensor, and the PDE reads
\begin{align*}
\p^2_t u &= \div \col{\t_1\\\t_2\\\t_2} = \p_{x_1}\t_1 + \p_{x_2}\t_2 + \p_{x_3}\t_3.
\end{align*}
 Since we are only interested in a principal symbol calculation (this is enough since obtaining the lower order terms is quite standard in the literature) we sometimes replace tangential derivatives, $\p_t,\p_{x_1},\p_{x_2}$ by $-i\tau,-i\xi_1,-i\xi_2$ respectively to ease the notation.
Then the PDE can be put into the form
\begin{equation}
\p_{x_3}\col{u\\\t_3} = A(x,D'_x,D_t)\col{u\\ \t_3} = \bmat{a_{11} & a_{12} \\ a_{21} & a_{11}^T} \col{u\\\t_3}.
\end{equation}
Here, we have principal symbols
\begin{align*}
 a_{11} &= \frac{1}{i}\bmat{0&0&\xi_1\\0&0&\xi_2\\\alpha \xi_1&\alpha \xi_2& 0},
\qquad a_{12} = \bmat{\mu^{-1} &0&0\\0&\mu^{-1}&0\\0&0&(\lambda + 2\mu)^{-1}},
 \\
a_{21} &= \bmat{ \beta_1\xi_1^2 + \mu \xi_2^2 - \tau^2 & \xi_1\xi_2\beta_2&0\\
\xi_1\xi_2 \beta_2 & \mu \xi_1^2 + \beta_1\xi_2^2-\tau^2 & 0 \\ 0 &0&-\tau^2}
\end{align*}
where
$$ \alpha = \frac{\lambda}{\lambda+2\mu}, \qquad \beta_1 = 4\mu \frac{\lambda+\mu}{\lambda+2\mu}, \qquad\beta_2 = \mu\frac{3\lambda+2\mu}{\lambda+2\mu}.$$

If $u^{(j)}$, $j=1,2$ represents $u$ on each side of the interface, then the interface conditions become simply
\begin{align*}
u^{(1)} &= u^{(2)} \quad \text{ on } \Gamma\\
\mathbf t_3^{(1)} &= \mathbf t_3^{(2)} \quad \text{ on } \Gamma.
\end{align*}
Let us denote $U= [u\;\;\; \mathbf t_3]^\top$.
As shown in \cite[section 3]{StolkThesis}\cite{HUPolarization}, there is an elliptic matrix pseudodifferential operator in $\Psi^0$ denoted $S(x,D',D_t)$ with microlocal inverse $S^-$ that diagonalizes $A$:
\[
A(x,D',D_t) = S(x,D',D_t) \text{diag}(C^+_p,C^+_s,C^-_p,C^-_s) S^-(x,D',D_t)
\]
where $C_{p}^\pm$ are scalar operators in $\Psi^1$ corresponding to the incoming and outgoing $P$ waves, and $C_s^\pm$ is a diagonal $2 \times 2$ pseudodifferential operator corresponding to the incoming and outgoing $s$ waves.

We can denote the columns of $S$ by $S_{\PS,\pm}$ in correspondence with the diagonal matrix, so that the modes are exactly $V_\mu := (S^-_{\mu a}U_a^{\ })_{a=1}^6$ where $S^-_{\mu a}$ is a parametrix for $S_{\mu a}$, $\mu$ will correspond to a $(\PS,\pm)$ pair and the subscript $\mu a$ refers to a matrix entry. Denote by $S^{(1)}, S^{(2)}$ the matrix $S$ on either side of the interface. We will refer to the plus $(+)$ as the in-modes, or modes for which the amplitude is known, that is, the incoming hyperbolic and the growing elliptic modes, which are the first $3$ columns of $S$. The minus $(-)$ will then be the outgoing modes and are the last $3$ columns. This indexing of $(\pm)$ here is unrelated to the $\pm$ indexing we used earlier to describe the opposite sides of an interface.

With this notation, the interface conditions read
\[
  S^{(1)}V^{(1)} = S^{(2)}V^{(2)}  \text{ on } \Gamma.
\]
 The first three components of $V^{(1)}$, say $v_I^{(1)}$, represent the incident $\PS$-modes, and the interface conditions become satisfied by choosing
\[
 V^{(1)} = \col{ v_I^{(1)}\\ R(x',D',D_t)v_I^{(1)}}, \qquad  V^{(2)} = \col{ T(x',D',D_t)v_I^{(1)}\\ 0 }
\]
for some ``reflection'' and ``transmission'' operators $R,T \in \Psi^0(\RR_t \times \Gamma; \mathbb{C}^3)$.

Thus, we obtain
\[
\col{ v_I^{(1)}\\ Rv_I^{(1)}} = (S^{(1)})^{-1}S^{(2)}\col{ Tv_I^{(1)}\\ 0 } =: Q\col{ Tv_I^{(1)}\\ 0 }.
\]
So writing $Q = \bmat{Q_{11} & Q_{12} \\ Q_{21}&Q_{22}}$ we obtain the two equations
\begin{align*}
I = Q_{11}T \text{ and } R = Q_{21}T
\end{align*}
If we can show that $Q_{11}$ is microlocally invertible, we would obtain $T = Q_{11}^{-1}$ and $R = Q_{21}Q_{11}^{-1}$. We will show in Appendix \ref{app: refl/trans} that $R$ and $T$ (and hence $M_R,M_T$ introduced before) are actually elliptic except on the joint $P$ and $S$ elliptic set on both sides of the interface. Note that $M_R,M_T$ only differ by $R,T$ by applications of elliptic operators, so ellipticity is not affected. Specifically, one has $M_R = S^{(1)} R (S^{(1)})^{-1}$ microlocally and likewise for $M_T$ using $S^{(2)}$.

\subsection*{Significance and list of operators}
Since we have so many symbols and operators, let us summarize them for quick reference.

\begin{center}
	\renewcommand\arraystretch{1.2}
    \begin{tabular}{c L{3.5cm} L{11.5cm}}
    \toprule
    \bfseries Operator & \bfseries Name & \bfseries Summary \\
    \midrule
    $J\CtoS$ &Cauchy to solution operator&
    Propagator mapping Cauchy data to the corresponding solution of the homogeneous elastic wave equation. \\
    $J_{\mathbf{C} \to \mathbf{S}+}$&
    forward Cauchy to solution operator&
    Similar to $J\CtoS$, but only propagates waves forward in time.
    \\
    $J\CtoB$&
    Cauchy to boundary map&
    Restriction of $J_{\mathbf C \to\mathbf S}$ to the boundary, which includes each side of an interface.
    \\
    $J_{\mathbf{C}\to \p+}$&
     forward Cauchy to boundary map&
     As $J\CtoB$, but with only waves that travel forward in time.
    \\
    $J\BtoS$ &
    boundary to solution map &
    Maps boundary data (associated with specific side of an interface) to a wave solution in the interior, traveling forward in time.
    \\
    $\Pi_{\PS}$ &
    $\PS$ projectors&
    Microlocal projectors of an elastic wavefield $u$ onto the $\PS$-characteristic set.
    \\
    $J\BtoB$ &
    boundary to boundary map&
    Restriction of $J\BtoS$ to the boundary (which includes interfaces). Hence, it propagates boundary data to the next boundary that the waves intersect.
    \\
    $M_{R/T}$ &
     reflection and transmission operators&
     Zeroth-order PsiDO's at the boundary that act as the reflection/transmission coefficients of the scattered wave from an incident field at an interface.
    \\
    \bottomrule
    \end{tabular}
\end{center}

The construction of the parametrix is now taken directly from \cite[Appendix]{CHKUControl}.\\

\subsection{Parametrix} \qquad First it will be convenient to define $M = M_R + \iota M_T$. With all the necessary components defined, we now set
\begin{align}
\tilde F &= J\CtoS + J\BtoS \sum_{k=0}^\infty(J\BtoB M)^kJ\CtoB \\
\tilde R_{2T} &= r_{2T} \circ \tilde F,
\end{align}
where $r_{2T}$ is restriction to $t=2T$. Again omitting the proof, it can be shown that $\tilde F \equiv F$ and $\tilde R_{2T} \equiv R_{2T}$ away from glancing rays. In the elastic case it means away from both $P$ and $S$ glancing rays; that is, for initial data $h_0$ such that every broken bicharacteristic originating in $\text{WF}(h_0)$ is sufficiently far from glancing. Recalling that $M = M_R + M_T$, we may write $\tilde R_{2T}$ as a sum of graph FIO indexed by sequences of reflections and transmissions:
\begin{align}
&\tilde R_{2T} = \sum_{ s \in \{R,T\}^k} \tilde R_s,
\qquad \qquad
\tilde R_{()} = r_{2T}J\CtoS \\
&\tilde R_{(s_1,\dots,s_k)} = r_{2T}J\BtoS M_{s_k}J\BtoB \cdots M_{s_2}J\BtoB M_{s_1}J\CtoB.
\end{align}
The solution operator $\tilde F$ likewise decomposes into analogous components $\tilde F_a$.

Now that we have a parametrix, we will use it in the next section to obtain certain ``subsurface'' travel time and lens relations.

\section{Proof of Theorem \ref{thm: main theorem}}
In this section, we will prove our main result on the uniqueness of elastic wave speeds under the foliation condition.

A key ingredient in the proof of uniqueness will be the following theorem proved by Stefanov, Uhlmann, and Vasy in \cite{UVLocalRay}.
\begin{theorem}\label{thm: UV local rigidity}
Choose a fixed metric $g_0$ on $\Omega$. Let $n= \text{dim}( \Omega) \geq 3$; let $c,\tilde c >0$ be smooth, and suppose $\p \Omega$ is convex with respect to both $g = c^{-2}g_0$ and $\tilde g = \tilde c^{-2}g_0$ near a fixed $p \in \p \Omega$. If $d_g(p_1,p_2)= d_{\tilde g}(p_1,p_2)$ for $p_1,p_2$ on $\p \Omega$ near $p$, then $c= \tilde c $ in $\Omega$ near $p$.
\end{theorem}

We write down a trivial corollary due to continuity of the distance function.
\begin{cor}\label{cor: rigidity from dense set of point}
Consider the same setup as in the above theorem. If $d_g(p_1,p_2)= d_{\tilde g}(p_1,p_2)$ for a dense set of points $p_1,p_2$ on some neighborhood of $P$ in $\p \Omega$, then $c= \tilde c $ in $\Omega$ near $P$.
\end{cor}
We need this since due to the multiple scattering in our setting, we will only be able to recover boundary travel times on a dense set of points and not a full neighborhood.

\subsubsection*{Summary of the proof of Theorem \ref{thm: main theorem}}

The proof of the main theorem is technical but the main argument is quite intuitive and geometric. Thus, we provide a summary of the proof that emphasizes the key ideas. Inductively, suppose that we have recovered the Lam\'{e} parameters above $\Sigma_\tau$, $\tau >0$, that is, inside $\Omega_\tau^c$, and let $z \in \Sigma_\tau$. Say we want to use Theorem \ref{thm: UV local rigidity} to recover $c_P$ near $z$ (a similar argument works for $c_S$). Viewing $\Sigma_\tau$ as the boundary of the domain $\Omega_\tau$, we would need to recover the local boundary distance function $d_P|_{\Sigma_\tau \times \Sigma_\tau}$ near $z$ to apply the above theorem. Let $x \in \Sigma_P$ near $z$ where $\Sigma_P$ is the $P$-characteristic set defined earlier, $(x,\xi) \in S_{\Sigma_\tau}^*\overline{\Omega}_\tau$ pointing downward, and $l_{P,\tau}(x,\xi)$ the corresponding boundary travel time that we would like to recover. 

Let $\gamma$ be a purely transmitted $P$-bicharacteristic, entering $\Omega$ at some time $t<0$ and passing through $(x,\xi)$ at time $t=0$. For convenience, let us view $\gamma$ as lying in $T^*(\RR^3 \times \RR_t)$. With appropriate Cauchy data $h_0$ supported outside $\Omega$, we can generate a microlocal $P$-wave whose wavefront set is initially along $\gamma$. Let us denote this wave solution by $u_{h_0}$. By propagation of singularities, $u_{h_0}$ will have each point along $\gamma$ in its wavefront set and two points in particular: $\gamma(0)$, which projects to $(x,\xi)\in T^* \RR^3$, and $\gamma(l_{P,\tau}(x,\xi))$. The problem is that due to the interface and the multiple scattering of $\PS$-waves in the interior that will also lie in $\text{WF}(u_{h_0})$, we cannot uniquely recover $\gamma(l_P(x,\xi))$ in this wavefront set. 
Hence, in addition to $h_0$, we must microlocally construct additional Cauchy data that eliminates this type of multiple scattering. The $h_0$ we construct will in fact be slightly more complicated since we do not want the mode converted transmissions resulting from the initial $P$-wave (see Figures \ref{f: subsurface ray} and \ref{f: multiples}). After this additional ``tail'' is constructed, we will be able to uniquely identify $\gamma(l_{P,\tau}(x,\xi))$ in the wavefront set.

We now turn to the multiple lemmas and propositions involved in proving the main theorem.\\

\begin{figure}
\centering
             \includegraphics[scale=0.8,page=2]{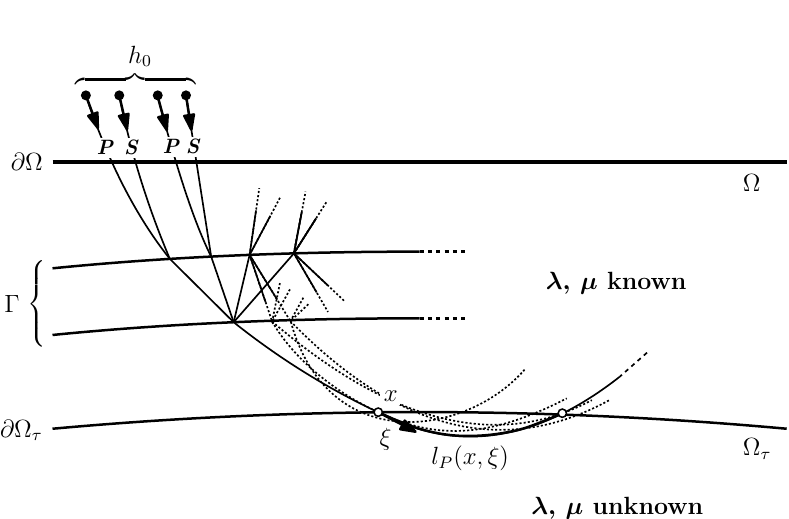}
\caption{With knowledge of the elastic parameters above $x$, it is possible to construct initial data $h_0$ that produces a single $P$ or $S$ ray at almost every covector $(x,\xi)$, here a $P$ ray. However, due to the presence of multiple reflected rays, it is not immediately possible to recover the length $l_{P,\tau}(x,\xi)$.}
\label{f: subsurface ray}
\end{figure}

\begin{figure}
\centering
             \includegraphics[scale=0.8,page=3]{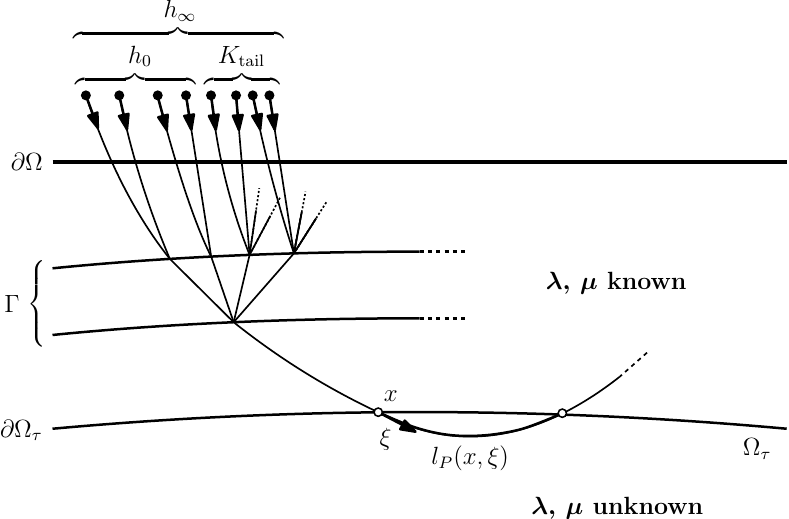}
\caption{By appropriately augmenting the initial data $h_0$ with extra initial data $K\tail$, producing total initial data $h_\infty$, multiple reflections can be suppressed, and $l_{P,\tau}(x,\xi)$ can be recovered from examination of the solution's wavefront set.}
\label{f: multiples}
\end{figure}

Let $\mathcal S \subset T^*\Omega$ be the set of $\xi$ such that every bad bicharacteristic through $\xi$ is \emph{(+)-escapable} (all definitions are in Appendix \ref{s: proving internal source construction}). The set $\mathcal S$ will be dense within an appropriate set, allowing us to work wholly inside $\mathcal S$ (see Lemma \ref{l: mathcal S is big enough}).
We will state a series of propositions and lemmas to prove the main theorem, some of which we prove in Appendix~\ref{s: proving internal source construction}. We first state the following crucial proposition that is at the heart of proving our main theorem and whose proof requires the microlocal analysis of scattering control.
\begin{prop}\label{prop: internal source construction}
Let $(x,\xi) \in \mathcal S$, $\tau=\rho(x)$, and let $v$ be a distribution whose wavefront set is exactly $(x,\RR_+\xi)$. Then there exists Cauchy data $h_\infty$ supported outside $\Omega$ and a large enough time $T>0$ such that $\text{WF}(R_T h_\infty) = \text{WF}(v)$ and
$\text{WF}(R_{T+s} h_\infty) = \text{WF}(R_s v)$ inside $\Omega_\tau$ for all $s \geq 0$.
Moreover, we may arrange that $\text{WF} (u_{h_\infty}) \subset \Sigma_P $ within $\Omega_\tau$ for times $t$ close enough to $T$. The same may be done with $\text{WF}(u_{h_\infty}) \subset \Sigma_S$ instead.
\end{prop}

\begin{rem}
The time $T$ is very concrete. It is essentially a scalar multiple of the $S$-distance from $x$ to $\p \Theta$. The reason is that we need access to all $S$ wave constituents starting near $x$ that produce branches that eventually return to the surface $\p \Theta$. The details will be made clear in the proof.
\end{rem}

\begin{rem}
The second part of the theorem means that not only can we generate Cauchy data to produce a certain singularity at a given depth, but we may even construct it to be a $P$ or an $S$ wave. This is essential for the uniqueness result since we must be able to recover subsurface lens relations for the $P$ and $S$ speeds separately.
\end{rem}
\emph{(Proof sketch)}
Let $\gamma$ be a purely transmitted $P$-bicharacteristic starting at $\p \Theta$ (when projected to the base space) for $t=0$  and $\gamma(T) = (x,\xi)$. Let $h_0$ be any Cauchy data supported in $\Theta \setminus \Omega$ with wavefront set containing $\RR_+\gamma(0)$ so that inside $\bar \Theta_T$,
$ \text{WF}(R_T h_0) = \text{WF}(v)$ by finite speed of propagation (see Figure \ref{f: subsurface ray}). 
The key now is to construct a tail that eliminates the multiple scattering and justify that such an $h_0$ above can be constructed. See Appendix \ref{s: proving internal source construction} for the remainder of the proof.\\

The next two lemmas are the main technical complications in the elastic setting. When we later show uniqueness via layer stripping, we will be able to layer strip past an interface if the wavespeeds of both Lam\'{e} systems infinitesimally match up just past the interface, even when we do not have direct access past such an interface. To do this, we rely on obtaining the principal symbols of reflection coefficients to recover the infinitesimal jumps in wave speeds past the interfaces. In terms of notation, any symbol with a tilde above it represents the corresponding symbol for the second set of Lam\'{e} parameters, and the superscript \emph{prin} denotes the principal symbol of a pseudodifferential operator.

\begin{lemma}\label{l: recover refl coeffs}
Suppose that $\Sigma_\tau \subset \Gamma$ and $c_{\PS} = \tilde c_{\PS}$ outside $\overline{\Omega}_\tau.$ Assume $\mathcal F = \tilde{\mathcal F}$. Then
\[M\prin_R = \tilde {M}\prin_R
\text{ on } T^*\Sigma^-_\tau.
\]
\end{lemma}
\begin{rem}
The proof actually shows that one may recover the full symbol, but it is unnecessary in our analysis.
\end{rem}
\proof This is essentially an inductive argument, whereby we recover the coefficients at each successive interface using appropriate sources. Let $\Sigma_{\tau_1}$ denote the first interface, and suppose $c_{\PS}=\tilde c_{\PS}$ outside $\Omega_{\tau_1}$.
Since $\rho = 1$, both elastic operators $\Op,\tilde \Op$ (see \S\ref{s: basic setup} for notation) agree on $\Omega_{\tau_1}^\star$. Combining this with $\mathcal F = \tilde{\mathcal F}$, propagation of singularities, and the convex foliation assumption to ensure no trapped rays, then $u_h \equiv \tilde u_h$ in $\Omega_{\tau_1}^\star$ for $h \in \mathbf C$.\footnote{In fact, we can use unique continuation to obtain the same result since we are allowed to measure outside $\Theta$ for an unlimited amount of time. Nevertheless, this is overkill for what we need here, which is a microlocal equivalence.} By taking a limit to $\Sigma_{\tau_1}^-$, we get $M_R J\CtoB h = \tilde{M}_R J\CtoB h$. By considering $h$ to be a $P$ wave and then picking $h$ to be an $S$ wave, we obtain the desired claim since we only need $M_R$ in the hyperbolic regions where $J\CtoB$ is elliptic and so we may generate microlocal $P$ and $S$ waves at the first interface. The argument is a direct analogue to the one in \cite[section 2.3]{RachBoundary}.

To proceed inductively, suppose $M_R$ is recovered for the first $k$ interfaces $\Sigma_{\tau_1},\dots,\Sigma_{\tau_{k-1}}$. Let $\Sigma_{\tau_k}$ be the $k^{\text{th}}$ interface and let $(y,\eta) \in \p^+S^*\Omega_{\tau_k}$ be a fixed covector. We assume $c_{\PS} = \tilde c_{\PS}$ in $\Omega_{\tau_k}^\star$ and so we may assume the transmission coefficients are recovered for these interfaces as well. We let $(x_0,\xi_0) \in T^*\Theta^\star$ lie on the same $P$-transmitted ray as $(y,\eta)$ which exists due to the convex foliation. We will repeat this construction for the $S$-transmitted ray too. Let $h$ be Cauchy data supported in $\Theta^*$ whose wavefront set in $S^*\RR^3$ is exactly $(x_0,\xi_0)$. The constituent of $\mathcal F h|_{\p \Theta}$ associated to the first primary reflection from $\Sigma_{\tau_k}$ is
\[
 M_R (J\BtoB M_T)^{k-1} J \BtoB J \CtoB h.
\]
Due to the convex foliation assumption, our assumptions on the wave speeds, and that $\mathcal F h = \tilde{\mathcal F}h$, we again have $u_h \equiv u_{\tilde h}$ on $\Omega^\star_\tau$ by propagation of singularities. Hence, the associated constituent for $\tilde {\mathcal F h}$ must be equal to this one at $\Sigma_\tau$ since we are not looking at what happens inside $\Omega_\tau$, as we are only considering a reflection. Since $M\prin_T$ are the same for both operators on $\Sigma_{\tau_j}$, $j=1,\dots,k-1$, by our assumption on the wavespeeds, then the same argument as before where we let $h$ generate $s$ waves associated to a purely transmitted $s$-ray through $(y,\eta)$ shows $M_R\prin(y,\eta',\tau) = \tilde M\prin_R(y,\eta',\tau)$ by applying the inverse of $M_T\prin$ and of $J\BtoB\prin$. Here, $(y,\eta')$ is the projection of $(y,\eta)$ to $T^*\Sigma_{\tau_k}$. We are using the fact that, since the Lam\'{e} parameters match on $\Omega^\star_{\tau_k}$, the operators $J\BtoB$ are equal as well for operators in this region. Also, these operators are elliptic near the hyperbolic point sets we are considering.
$\Box$

\begin{lemma}\label{l: jumps from reflection coeff}
Suppose that $\Sigma_\tau \subset \Gamma$, $c_{\PS} = \tilde c_{\PS}$ outside $\overline{\Omega}_\tau$, and denote $\Sigma_\tau^{\pm}$ for the two sides of $\Sigma_\tau$. Suppose that $M\prin_R = \tilde M\prin_R$ on $T^*(\RR_t\times\Sigma^-_\tau)$ for at least two linearly independent covectors. Then $c_{\PS} = \tilde c_{\PS}$ on $\Sigma_\tau^+$.
\end{lemma}

The above lemma is essentially saying that the principal symbols of reflection coefficients are enough to recover the jumps in both wave speeds at an interface. This should not come as a surprise since the reflection coefficient would vanish identically if the speeds were actually continuous across the interface. Thus, two waves with identical reflections, must also have transmissions that correspond to the same covectors. Since the proof of the lemma is quite technical, we save it for the appendix (see Appendix \ref{appendix: proof of jumps from refl coeff}).

Both of these crucial lemmas are basically all that is necessary to recover subsurface travel times and lens relations for a particular covector.
We will show the following:
Let $\Gamma \subset \Sigma_\tau$ be relatively open and let $T>0$. Then the lens relations $(L_{P,\tau},l_{P,\tau})$ and $(L_{S,\tau},l_{S,\tau})$ are determined uniquely on the open sets of $(x,v)$ with $ x \in \Gamma$ so that the unit speed geodesic issued from $(x,v)$ at time $0$ in the metric $c_P^{-2}dx^2$, respectively $c_S^{-2}dx^2$, is transversal at $x$ and hits $\Sigma_\tau$ again, transversely, at a point in $\Gamma$ at a time not exceeding $T$ and without hitting any other interfaces. Since we assume that the interfaces are not dense, one may always ensure with $T$ or $\Gamma$ small enough that such rays do not hit another interface before returning to $\Sigma_\tau$.

Also, to recover the lens relation for a particular covector, we will need to use the microlocal scattering control in the form of Proposition \ref{prop: internal source construction}. This requires covectors belonging to $\mathcal S$ and we must ensure there are enough of them. The following lemma uses convex foliation to ensure that we have enough of them.

\begin{lemma}\label{l: mathcal S is big enough}
Let $x \in \Sigma_\tau$ for some $\tau$. Then there is a neighborhood $B_x \subset \Sigma_\tau$ such that $B_x \cap \p^+ S^* \Omega_\tau \cap \mathcal S$ is dense in $B_x \cap \p^+ S^* \Omega_\tau$.
\end{lemma}

\begin{proof}
The proof follows from the convex foliation condition and repeated application of Lemma \ref{l: purely transmitted rays are dense} and its proof.

Take a particular covector $(x,\nu) \in \p ^+S^* \Omega_\tau$ pointing upwards and let $\gamma_{\PS, \nu_x}$ be the associated smooth bicharacteristic starting at $(x,\nu)$. Considering $\gamma_{P,\nu_x}$ first, it will either glance or hit the next interface $\Sigma_{\tau_1}$ at time $t_1$, say, transversely. If the latter, the convex foliation guarantees that both the $P$ and $S$ transmitted branches continuing $\gamma_{P,\nu_x}$ will also be transverse to $\Sigma_{\tau_1}$ and move ``upward'' (decreasing $\rho$). Also, there will be exactly two opposite branches at $\gamma_{P,\nu_x}(t_1)$ that are transverse to $\Sigma_{\tau_1}$ and move upward in backward time. If it glances, then by Lemma \ref{l: glancing rays are small}, an arbitrary perturbation of $\nu$ avoids this. We can apply this analysis to each successive $P$ branch discussed and iterate; since the time $T_s$ in the definition of escapability is finite, there will be only finitely many branchings and so there will be a dense set of $\nu \in \p^+ S_x^*\Omega_\tau$ such that all the $P$-branches of $\gamma_{P,\nu_x}$ escape. The continued $S$ branches will be analyzed next.

Let us now consider $\gamma_{S,\nu_x}$ and use the same notation $t_1$ and $\Sigma_{\tau_1}$ as in the previous case. The analysis for $\gamma_{S,\nu_x}$ will apply just as well for the $s$ branches discussed in the previous paragraph. If $\gamma_{S,\nu_x}(t_1)$ does not glance, the convex foliation guarantees an $S$ transmitted branch that continues $\gamma_{P,\nu_x}$, is transverse to $\Sigma_{\tau_1}$, and moves ``upward'' (decreasing $\rho$). The issue is that the transmitted $P$ branch might be glancing where we have hit a critical angle. However, this glancing set is a dimension lower than the hyperbolic points and so we may perturb this $P$-branch to be transversal to $\Sigma_{\tau_1}$ and move upward. We may then continue this branch backward with an $S$ ray that starts on $\p^+ S^*\Omega_\tau$, is a slight perturbation of $\gamma_{S,\nu_x}$, and has a different base point.

 Hence, we now have both a transmitted $P$ and $S$ branch moving upwards by convex foliation, and an opposite $P$ and $S$ branch moving upwards backward in time. We then apply the analysis in the last paragraph and iterate the above for each successive interface. Hence, either $(x,\nu)$ or an open set of perturbations of it will be escapable.

Using Lemma \ref{l: glancing rays are small}, the above analysis shows there is a neighborhood $B_x \subset \Sigma_\tau$ of $x$ such that a dense set of $\p^+ S_{B_x}^*\Omega_\tau$ are escapable. Indeed, any covector that is not escapable can be perturbed by the above procedure.
\end{proof}

In the following series of proofs, we rely on the previous lemma to keep using Proposition \ref{prop: internal source construction} without explicitly saying so.

\begin{lemma}\label{l: lens relation recovery}
Let $(x,\xi) \in \p T^* \Omega \cap S^*_+\Omega$ as described above and assume the convex foliation condition. If $\mathcal F = \tilde {\mathcal F}$ and $\lambda = \tilde \lambda$, $\mu = \tilde \mu$ outside $\Omega_\tau$, then $c_{\PS}$ and $\tilde c_{\PS}$ have identical subsurface lens relations w.r.t. $\Sigma_\tau$ in a neighborhood of $(x,\xi)$ within $T_{\Sigma_\tau}^*\Omega$.
\end{lemma}

\proof
Without loss of generality, under the convex foliation condition we may assume that $x$ is a regular point since otherwise, one may use a continuity/density argument described in \cite{CHKUUniqueness}. We will divide the proof into two cases, which have slightly different proofs.

\medskip

\textbf{The point \boldmath $x$ is not on an interface:}
We let $v \in \mathcal E'(\Omega)$ be such that $\text{WF}(v) = (x,\RR \xi)$ and let $h$ be as in Proposition \ref{prop: internal source construction} supported outside $\Omega$. We let $u = F(R_{-a}h)$ and $\tilde u = \tilde F(R_{-a} h)$ with an appropriately chosen $a$ based on the support of $h$ (see \cite{CHKUUniqueness} for details). Now, $u = \tilde u$ outside $\Omega$ and by unique continuation, $u= \tilde u$ outside $\Omega_\tau$ since the Lam\'{e} parameters coincide there. In fact all we need is that $u \equiv \tilde u$ in $\Omega_\tau^\star$ which follows by microlocal analysis. Indeed, any ray in this set has a branch that escapes $\Omega$ by the foliation condition. Thus, by propagation of singularities, $u=\tilde{u}$ inside $\Omega_\tau^\star$ modulo smoothing.

 Let $T$ denote the time the transmitted geodesic from $\p \Omega$ reaches $\xi$. By Proposition \ref{prop: internal source construction}, we can ensure $\text{WF}(u)$ restricted to $\RR_t \times \Omega_\tau$ is generated purely from the $P$-ray associated to $\xi$. We only consider those $\xi$ whose associated $P$-geodesic does not encounter any interface before reaching $\p \Omega_\tau$. This is always possible by the convex foliation condition and taking $\xi$ that are near tangent to $\p \Omega_\tau$. Since the Lam\'{e} parameters are smooth near $x$, then for a $\xi$ nearly tangential to $\p \Omega_\tau$, the first singularity of $u$ in $T^*_+ \Sigma_\tau$ occurs at time $T + l_{P,\tau}(x,\xi)$ and covector $L_{P,\tau}(x,\xi)$. This must be true for $\tilde u$ as well since $u=\tilde u$ outside $\Omega_\tau$. Hence, $l_{P,\tau}(x,\xi) = \tilde l_{P,\tau} (x,\xi)$ and $L_{P,\tau}(x,\xi) = \tilde L_{P,\tau} (x,\xi)$. We then repeat the above argument using Proposition \ref{prop: internal source construction} to generate a pure $s$-wave, singular precisely at $(x,\xi)$ when restricted to $\Omega_\tau$ at the appropriate time. This works since we can always restrict to rays which do not hit any interfaces before returning to $\Sigma_\tau$ by the convex foliation.
\medskip

\textbf{The point \boldmath $x$ is at an interface:}
First, without loss of generality, we may assume that $\Sigma_\tau$ actually coincides with the interface near $x$. Indeed, any point $z \in \Sigma_\tau$ near $x$ that is not at an interface implies that the Lam\'{e} parameters are smooth there. Hence we may apply the above result for the smooth case combined with Theorem \ref{thm: UV local rigidity} to show that the wavespeeds coincide near such points. Progressing in this fashion shows that both wave speeds in fact coincide near $x$ up to the interface that contains $x$, and so the wavefields coincide there as well. Hence, we may assume that $\Sigma_\tau$ is the interface.
\\

 Using Lemma \ref{l: jumps from reflection coeff}, we conclude that if $u$ is a pure $P$-wave for some time in $\Omega_\tau$, then $\tilde u$ is as well, both associated to $(x,\xi)$, even though inside $\Omega_\tau$ they could theoretically be quite different.

We then examine the construction of $h$ in Proposition \ref{prop: internal source construction} more closely.\footnote{The following argument is necessary to ensure that we match a $P$ travel time associated to $\Op$ with the corresponding one associated to $\tilde \Op$ rather than an $s$ travel time associated to $\tilde \Op$.} The $\PS$-directly transmitted component of $R_T h$ is $\mathbf {DT}^+_{k,\PS} h$ from definition \ref{def: microlocal direct trans}.
We make the decomposition $h= h_0 + K\tail$. We take any wavefield $v$, supported in $\Omega_\tau$ initially and whose wavefront set is exactly $\gamma^P_{\xi_x}$ and inside $\Sigma_P$, where $\gamma^P_{\xi_x}$ is a $P$-bicharacteristic whose initial covector is $(x,\xi) \in T^*\Omega_\tau$. With $\rho_\Gamma$ denoting restriction to $\Gamma$, we may view $\rho_{\Sigma_\tau} v$ as boundary data. The construction of $h_0$ and $K_{tail}$ in Proposition \ref{prop: internal source construction} ensures
\[
R_{T+t} h|_{\Omega_\tau} \equiv \mathbf{DT}_{k,p}^+h_0|_{\Omega_\tau}
\equiv v|_{\Omega_\tau}.
\]
That is, the directly transmitted constituent of $h$ inside $\Omega_\tau$ (the ``underside'' of $\Sigma_\tau$) is precisely a $P$-wave associated to $(x,\xi)$. The point is that the same initial data $h$ will also produce a pure $P$-wave with respect to $\tilde c_P$ on the underside of $\Sigma_\tau$ by Lemma \ref{l: jumps from reflection coeff} since that lemma implies that $\tilde M_T \equiv M_T$ at $T^*\Sigma_\tau$ near $x$.

Thus, since the transmission matrices of $u$ and $\tilde u$ coincide microlocally near $(x,\xi)$, then
\[
\rho_{\Sigma_\tau} v \equiv \rho_{\Sigma_\tau} u|_{\Omega_\tau} \equiv \mathbf{DT}^+_{k,P} h = \tilde{\mathbf{DT}}^+_{k,P} h
\equiv \rho_{\Sigma_\tau} \tilde u|_{\Omega_\tau}.
\]
We note that inside $\Omega_\tau$, $\tilde J \BtoS \rho_{\Sigma_\tau} v$ is indeed a pure $P$-wave associated to $(x,\xi)$, so $\tilde J \BtoS \rho_{\Sigma_\tau} \tilde u$ will be as well with speed $\tilde c_P$.
By our assumptions, $\text{WF}(u|_{\Sigma_\tau^-}) = \text{WF}(\tilde u|_{\Sigma_\tau^-})$. By Proposition \ref{prop: internal source construction}, if we consider the $t$-component of this wavefront set, then the first $t$ past $T$ in this wavefront set will be precisely $l_{P,\tau}(x,\xi)$ by our construction. By equality of the wavefields and since $\tilde u$ was also a pure $P$-wave at time $T$ in $\Omega_\tau$, then $l_{P,\tau}(x,\xi)=\tilde l_{P,\tau}(x,\xi).$ A similar argument lets us conclude $l_{S,\tau}(x,\xi) = \tilde l_{S,\tau}(x,\xi)$ as well.
$\Box$
\newline

We can combine the above lemma with Theorem \ref{thm: UV local rigidity} to obtain the key corollary. First, let $d^\tau_{\PS}$ denote the $\PS$-distance function restricted to $\overline\Omega_\tau \times \overline\Omega_\tau$.
\begin{cor}\label{cor: local uniqueness}
With the assumptions in the above lemma, $d_{\PS}^\tau\big|_{\Sigma_\tau \times \Sigma_\tau} = \tilde d_{\PS}^\tau\big|_{\Sigma_\tau \times \Sigma_\tau}$ in some neighborhood of $x$, and $c_{\PS} = \tilde c_{\PS}$ in some neighborhood of $x$.
\end{cor}

\begin{theorem}\label{thm: uniqueness final thm}
Under the convex foliation condition, if $\mathcal F = \tilde {\mathcal F}$, then $c_{\PS} = \tilde c_{\PS}$.
\end{theorem}

\begin{proof}
The proof is by contradiction. Suppose $c_P \neq \tilde c_P$ or $c_S \neq \tilde c_S$, and let $f= |c_P - \tilde c_P|^2 + |c_S - \tilde c_S|^2$. Now consider $S:= \Omega_r \cap \text{supp} f$, and take $\tau = \min_S \rho$: 
so $c_{P} = \tilde c_{P}$ and $c_S = \tilde c_S$ above $\Omega_\tau$, but by compactness there is a point $x \in \Sigma_\tau \cap S$. The condition that $\rho^{-1}(T)$ has measure zero rules out the trivial case $\tau = \tau_0$.

Let us now consider a small neighborhood of $x$, denoted $B_x$, and we consider the $\Sigma_\tau$-boundary distance function $d^\tau_{\PS}$ restricted to such neighborhoods. Since the interfaces are not dense, and we assume convex foliation, we may choose $B_x$ small enough so that all $P$ and $S$ rays corresponding to rays staying completely inside $B_x$ do not reach an interface; i.e. even the mode converted rays do not reach an interface. This insure that a $P$-wave that hits $B_x$, transmits a $P$ and $S$ wave, the $P$-wave returns to $\Sigma_\tau$ first before any other ray.

We now consider two cases, depending on whether $x$ is on an interface of $c_{\PS}$ or not.
\medskip

\noindent\textbf{\boldmath Smooth case:} $x \notin \Gamma$. As in \cite{UVLocalRay} we use the fact that $c_{\PS}$ and $\tilde c_{\PS}$ are equal above $\Omega_\tau$ to show they locally have the same lens relation on $\Sigma_\tau$. We can then apply Corollary \ref{cor: local uniqueness} to show that in fact $c_{\PS} = \tilde c_{\PS}$ near $x$, contradicting $x \in \text{supp} f$. The additional wrinkle is that we must ensure that $\tilde c_{\PS}$ is also smooth near $x$. 

Suppose on the contrary that $\tilde c_\PS$ were not smooth at $x$. Since $c_\PS=\tilde c_\PS$ on $\Omega^\star_\tau$, if $x\in\tilde\Gamma$, then $\tilde\Gamma$ must be tangent to the leaf $\Sigma_\tau$. Now let $\gamma$ be any bicharacteristic through a covector $(x,\xi)$ not tangential to the leaf $\Sigma_\tau$, and choose initial data $h$ by Proposition~\ref{prop: internal source construction} satisfying $\WF(R_T h)=(x,\RR_+\xi)$. Then $\tilde u_h(T)$ is singular on the reflected bicharacteristic to $\gamma$ at $x$. But this is impossible, since $u_h(T)=\tilde u_h(T)$ on $\Omega^\star_\tau$, and the reflected bicharacteristic is contained in $\Omega^\star_\tau$ for $t$ slightly greater than $T$, since $\tilde\Gamma$ is tangent to $\Sigma_\tau$.

From the argument above, we conclude $c_\PS, \tilde c_\PS$ are smooth in a sufficiently small $\epsilon$-ball $B_\epsilon(x)$. Next, there exists a smaller neighborhood $B_{\epsilon'}(x) \subset B_\epsilon(x)$ in which every two points have a minimal-length path between them that is contained in $B_\epsilon (x)$, and in particular does not intersect $\Gamma \cup \bar \Gamma$. This is true by the boundedness of $c_\PS$ and $\tilde c_\PS$. Namely, picking global bounds $0<m<c_\PS,\tilde c_\PS < M$ and taking $\epsilon' = \epsilon m/(m+2M+1)$, one can verify $d_{\PS}(y,\p B_\epsilon(x))>2 \text{diam}_P \ B_{\epsilon'}(x)$.

Finally, we apply Lemma \ref{l: lens relation recovery}, concluding that $c_{\PS}$ and $\tilde c_{\PS}$ have identical lens relations for covectors $(x,\xi) \in \p T^*\Omega_\tau \cap T^*_-\Omega$ whose bicharacteristics do not intersect any interfaces before returning to $\Sigma_\tau$. Note that the lemma is applied multiple times to recover the lens relation for each wave speed. This is true, in particular, for the geodesics connecting points in $U = B_{\epsilon'}(x) \cap \Sigma_\tau$. Hence, $d^\tau_{\PS} = \tilde d^\tau_{\PS}$ on $U \times U$. Applying local boundary rigidity (Theorem \ref{thm: UV local rigidity} and its corollary), we conclude $c_P = \tilde c_P$  and $c_S = \tilde c_S$ on some neighborhood of $x$, contradicting $x \in \text{ess supp}\ f$.
\newline

\noindent\textbf{\boldmath Interface case: }$x \in \Gamma$. This follows from the above case using Lemma \ref{l: jumps from reflection coeff} and Lemma \ref{l: lens relation recovery}. Indeed, it is those two lemmas that allow us to recover the lens relation on the underside of the interface $\Sigma_\tau$ for both $c_P$ and $c_S$. Similarly to the smooth case, the ball $B_\epsilon(x)$ is constructed to be disjoint from any other interface except for $\Sigma_\tau$ so that rays between points in $\Sigma_\tau$, starting on the underside of $\Sigma_\tau$ in $B_\epsilon (x)$ stay completely in $B_\epsilon(x)$ before returning to $\Sigma_\tau$.
\end{proof}

\section{Discussion}\label{s: discussion}
In the proof of the theorem, we needed to assume the density $\rho =1$ in order that we could recover all Lam\'{e} parameters during the layer stripping, giving us access to the full wave solution in the known layers. In \cite{RachDensity}, Rachele shows how one may use ``lower order polarization'' data to recover the density $\rho$ as well under certain conditions. However, this was in the smooth setting and the result was global since it utilized a global inversion result of an X-ray transform of tensor fields. Since that paper, Stefanov, Uhlmann, and Vasy in \cite{SUVlocaltensor} have shown that one may also obtain local inversion results of the X-ray transform on tensors. Hence, it may be possible to combine Rachele's argument to obtain local, lower order polarization data containing information on the density from the outside measurement operator combined with the result in \cite{SUVlocaltensor} on the local ray transform on tensors to recover the density $\rho$ during our layer stripping procedure. We will pursue this strategy in another work.

\appendix

\section{Computation of reflection and transmission PsiDO's}
\label{app: refl/trans}
Let us recall the traction formulation of the elastic equation, and for simplicity, we assume a flat interface $\{ x_3 = 0\}$. The non-flat case will not require much more work, and we provide details on this later.

Set the unit normal to interface $\Gamma_i$ by $\nu = \col{0\\0\\1}$. Then the traction components are
$$ \t_j = \mathbb{C}\nabla_s u \cdot e_j$$
and the PDE reads
\begin{align*}
\p^2_t u &= \div \col{\t_1\\\t_2\\\t_2} = \p_{x_1}\t_1 + \p_{x_2}\t_2 + \p_{x_3}\t_3 \\
& \Leftrightarrow \p_{x_3}\t_3 = -i \xi_1 \t_1 -i \xi_2 \t_2 - \tau^2 u
\end{align*}
in the fourier regime.
Then the PDE can be put into the form
\begin{equation}
\p_{x_3}\col{u\\\t_3} = A(t,x',D_t,D_x')\col{u\\ \t_3} = \bmat{a_{11} & a_{12} \\ a_{21} & a_{11}^T} \col{u\\\t_3}.
\end{equation}
Here, we have principal symbols
\begin{align*}
 a_{11} &= \frac{1}{i}\bmat{0&0&\xi_1\\0&0&\xi_2\\\alpha \xi_1&\alpha \xi_2& 0}
\qquad a_{12} = \bmat{\mu^{-1} &0&0\\0&\mu^{-1}&0\\0&0&(\lambda + 2\mu)^{-1}}
 \\
a_{21} &= \bmat{ \beta_1\xi_1^2 + \mu \xi_2^2 - \tau^2 & \xi_1\xi_2\beta_2&0\\
\xi_1\xi_2 \beta_2 & \mu \xi_1^2 + \beta_1\xi_2^2-\tau^2 & 0 \\ 0 &0&-\tau^2}
\qquad a_{22} = a_{11}^T
\end{align*}
where
$$ \alpha = \frac{\lambda}{\lambda+2\mu} \qquad \beta_1 = 4\mu \frac{\lambda+\mu}{\lambda+2\mu} \qquad\beta_2 = \mu\frac{3\lambda+2\mu}{\lambda+2\mu}.$$

Also, the eigenvectors of $A$ are easy to find. Indeed one may define a $3\times 3$ matrix $\tilde{\t}_3 = \tilde{\t}_3(\xi_3)$ so that

 $$\t_3 = \tilde{\t}_3 u.$$
 Indeed, one merely takes
 \[
 \tilde{\t}_3 = i\bmat{\mu \xi_3 & 0 & \mu \xi_1 \\
 0& \mu \xi_3 & \mu \xi_2 \\
 \lambda \xi_1 & \lambda \xi_2 & (\lambda + 2\mu)\xi_3 }.
 \]
 Indeed, let $v$ be an eigenvector of the principal symbol $p$ so that
\begin{align*}
pv &= (\tau^2-c^2_{\PS}|\xi|^2)v \\
&\Leftrightarrow (\tau^2 \Id + i\xi_1 \tilde{\t}_1 +i\xi_2 \tilde{\t}_2 + i\xi_3\tilde{\t}_3)v = (\tau^2-c^2_{\PS}|\xi|^2)v \\
&\Leftrightarrow (\tau^2 \Id + i\xi_1 \tilde{\t}_1 +i\xi_2 \tilde{\t}_2 \pm i\xi_{3,\PS}\tilde{\t}_3)v = 0
\text{ if we set }\xi_3 = \pm \xi_{3,\PS} \\
\pm i\xi_{3,\PS}\tilde{\t}_3 v &=-( \tau^2 \Id + i\xi_1 \tilde{\t}_1 +i\xi_2 \tilde{\t}_2)v = a_{21}v + a_{22}\tilde{\t}_3(v).
\end{align*}
Thus, $\col{v \\ \tilde{\t}_3(v)}$ is an eigenvector of $A$ with eigenvalues $\pm i\xi_{3,\PS}.$

It will be useful to denote these normal momenta components by
\[
a_{\PS} = \sqrt{|\xi'|^2 - c_{\PS}^{-2}\tau^2}
\]
where we later use a superscript to distinguish which side of the interface we are considering.
Hence, we may form the $6 \times 6$ matrix of eigenvectors
\[
 S = \bmat{ |&|&|&|&|&| \\s_{P,+} & s_{sH,+}& s_{sV,+} & s_{P,-} & s_{sH,-}& s_{sV,-} \\ |&|&|&|&|&|}
\]
and corresponding eigenvalue diagonal matrix $\Lambda = \text{diag}(i\xi_{3,P},i\xi_{3,S},i\xi_{3,S},-i\xi_{3,P},-i\xi_{3,S},-i\xi_{3,S}).$
By construction, one has
$$ AS = S\Lambda.$$
Then form
$$ K = \bmat{ 0_3 & I_3 \\ I_3 & 0_3}.$$
Then we find $KA = (KA)^T$ by the symmetries of $A$. Hence
\begin{align*}
S^TKA &= S^T(KA)^T = (KAS)^T = (KS\Lambda)^T = \Lambda S^T K
\\
&\Rightarrow S^TKA = \Lambda S^T K \\
& \Rightarrow S^TKS\Lambda = S^TKAS = \Lambda S^T K S
\end{align*}
So $S^TKS$ commutes with a invertible diagonal matrix and hence must be diagonal. The first 3 columns will be a positive eigenvalue while the last 3 will be negative. Hence, if we rescale the columns of $S$ and define
$J = \bmat{I_3&0_3\\0_3&-I_3}$ then $JS^TKS$ will be the identity. Hence, under the rescaling we obtain

$$ S^{-1} = JS^TK$$
Then labelling $V_{\pm}$ as $3 \times 3$ matrices of eigenvectors, we have
$$ S = \bmat{ V_+ & V_-\\ \tilde{\t}_3V_+ & \tilde{\t}_3(V_-)}.$$
Hence,
$$ S^{-1} = \bmat{ \tilde{\t}_3(V_+)^T & V_+^T \\ -\tilde{\t}_3(V_-)^T & -V_-^T}.$$

More explicitly, one has
$$
V_{\pm} = \bmat{\xi_1 & \pm \xi_1a_S & -\xi_2 \\
            \xi_2 & \pm \xi_2 a_S & \xi_1\\
            \mp a_P &|\xi'|^2 & 0}
$$
And
\[
\tilde{t}_3 V_{\pm} = i\bmat{\mp 2\mu \xi_1 a_P & -2\mu \xi_1 \chi & \pm \mu a_S \xi_2 \\
                            \mp 2\mu \xi_2 a_P &  -2\mu \xi_2 \chi &  \mp \mu a_S \xi_1\\
                            2\mu \chi        &  \mp 2\mu |\xi'|^2a_S &   0
}
\]
where $\chi = \frac{1}{2}\tau^2/\mu - |\xi'|^2$.
Hence, we have
\[
S = \bmat{\xi_1 &  \xi_1a_S & -\xi_2  &    \xi_1 &  -\xi_1a_S & -\xi_2 \\
          \xi_2 &  \xi_2 a_S & \xi_1&  \xi_2 &  -\xi_2 a_S & \xi_1\\
          - a_P &|\xi'|^2 & 0        &     a_P &|\xi'|^2 & 0 \\

           -2\mu \xi_1 a_P & -2\mu \xi_1 \chi &  \mu a_S \xi_2 &  2\mu \xi_1 a_P & -2\mu \xi_1 \chi & -\mu a_S \xi_2 \\
 -2\mu \xi_2 a_P &  -2\mu \xi_2 \chi &   -\mu a_S \xi_1 &  2\mu \xi_2 a_P &  -2\mu \xi_2 \chi &   \mu a_S \xi_1\\
                            2\mu \chi        &   -2\mu |\xi'|^2a_S &   0 &2\mu \chi        &   2\mu |\xi'|^2a_S &   0
}
\]
And
\[
S^TK= \bmat{
-2\mu \xi_1 a_P & -2\mu \xi_2 a_P & 2\mu \chi &    \xi_1 & \xi_2 & -a_P \\
-2\mu \xi_1\chi &-2\mu\xi_2 \chi & -2\mu|\xi'|^2a_S & \xi_1 a_S&\xi_2 a_S&|\xi'|^2\\
\mu a_S\xi_2 & -\mu a_S \xi_1 & 0 & -\xi_2 &\xi_1 & 0 \\

2\mu \xi_1 a_P & 2\mu \xi_2 a_P & 2\mu \chi &    \xi_1 & \xi_2 & a_P \\
-2\mu \xi_1\chi &-2\mu\xi_2 \chi & 2\mu|\xi'|^2a_S & -\xi_1 a_S&-\xi_2 a_S&|\xi'|^2\\
-\mu a_S\xi_2 & \mu a_S \xi_1 & 0 & -\xi_2 &\xi_1 & 0
}
\]
A quick calculation shows
\[
S^TKS = \text{diag}(-2\tau^2a_P, -2\tau^2|\xi'|^2a_S,-2\tau^2a_S,2\tau^2a_P,2\tau^2|\xi'|^2a_S,2\tau^2a_S)=D.
\]
Thus, $S^{-1} = D^{-1}S^TK.$

Next is useful to define
\[
\tilde{E} = \bmat{ \xi_1 &\xi_2&0\\
0&0&1\\
-\xi_2&\xi_1&0}, \qquad
\tilde{E}^{-1} = \frac{1}{|\xi'|^2}\bmat{ \xi_1 &0&-\xi_2\\
\xi_2&0&\xi_1\\
0&|\xi'|^2&0}.
\]
And set $E = \bmat{\tilde{E}&0\\0&\tilde{E}}.$ Then
\[
ES = \bmat{|\xi'|^2 &|\xi'|^2a_S&0&|\xi'|^2&-|\xi'|^2a_S&0\\
-a_P&|\xi'|^2&0&a_P&|\xi'|^2&0 \\
0&0&|\xi'|^2&0&0&|\xi'|^2\\
-2\mu|\xi'|^2a_P& -2\mu|\xi'|^2\chi&0& 2\mu|\xi'|^2 a_p &-2\mu|\xi'|^2\chi&0 \\
2\mu\chi & -2\mu|\xi'|^2a_S & 0 &2\mu \chi & 2\mu|\xi'|^2a_S&0\\
0&0& -\mu|\xi'|^2a_S& 0& 0& \mu|\xi'|^2a_S
}
\]
Then
\[
(ES)^{-1} = S^{-1}E^{-1} = D^{-1}S^TKE^{-1}
\]
and
\[
|\xi'|^2 S^TKE^{-1}=
\bmat{
-2\mu|\xi'|^2 a_P& 2\mu |\xi'|^2\chi&0&|\xi'|^2&-a_P|\xi'|^2&0\\
-2\mu|\xi'|^2\chi&-2\mu|\xi'|^4a_S&0&|\xi'|^2a_S&|\xi'|^4&0\\
0&0&-\mu|\xi'|^2a_S&0&0&|\xi'|^2\\
2\mu|\xi'|^2 a_P& 2\mu |\xi'|^2\chi&0&|\xi'|^2&a_P|\xi'|^2&0\\
-2\mu|\xi'|^2\chi&2\mu|\xi'|^4a_S&0&-|\xi'|^2a_S&|\xi'|^4&0\\
0&0&\mu|\xi'|^2a_S&0&0&|\xi'|^2
}.
\]

Note that if we are away from normal incidence so that $|\xi'|^2\neq 0$, we may use this as an elliptic factor to make $S$ and $E$ order $0$ as long as we make $A$ order $1$. Then, we may remove all instances of $|\xi'|$ appearing in the above formulas and replace the appearance of $\tau$ with $\hat{\tau} = \tau /|\xi'|$.
Denoting the interface as $\Gamma$, we denote $U^{(i)} = S^{(i)}V^{(i)}$ as microlocal solutions to the PDE
with interface conditions given by
\[
   S^{(1)}V^{(1)} = S^{(2)}V^{(2)} \text{ on }\Gamma
\]
Now the components of $V^{(1)}$ represent an incident ``downgoing'' wave and the reflected ``upgoing'' wave
\[
 V^{(1)} = \col{ v_I^{(1)}\\ Rv_I^{(1)}}, \qquad  V^{(2)} = \col{ Tv_I^{(1)}\\ 0 }
\]
Thus, we obtain
\[
\col{ v_I^{(1)}\\ Rv_I^{(1)}} = (S^{(1)})^{-1}S^{(2)}\col{ Tv_I^{(1)}\\ 0 } := Q\col{ Tv_I^{(1)}\\ 0 }.
\]
So writing $Q = \bmat{Q_{11} & Q_{12} \\ Q_{21}&Q_{22}}$, where each entry is a $3 \times 3$ block matrix, we obtain the two equations
\begin{align*}
I = Q_{11}T \text{ and } R = Q_{21}T
\end{align*}
So if we have $Q_{11}$ being microlocally invertible, we would obtain $T = Q_{11}^{-1}$ and $R = Q_{21}Q_{11}^{-1}$.
Notice that $(S^{(1)})^{-1}S^{(2)}= D^{-1}((S^{(1)})^TKE^{-1})(ES^{(2)})$ so it will suffice to show that
$[((S^{(1)})^TKE^{-1})(ES^{(2)})]_{11}$ (the first $3\times 3$ subblock) is invertible.

By looking at the structure of $ES$ and $S^TKE^{-1}$, namely that each of the four sublocks have a block structure consisting of a $2 \times 2$ matrix, and a $1 \times 1$ matrix, and the $1 \times 1$ pieces are trivial, it will suffice to analyze the remaining $2 \times 2$ constituents.
Then the first $2\times2$ minor of this matrix is given by the multiplication of
\[
\bmat{-2\mu_1a_P^{(1)}&2\mu_1 \chi^{(1)}&1&-a_P^{(1)} \\
-2\mu_1\chi^{(1)}&-2\mu_1a_S^{(1)}&a_S^{(1)}&1
}
\bmat{
1&a_S^{(2)}\\
-a_P^{(2)}&1\\
-2\mu_2a_P^{(2)}&-2\mu_2\chi^{(2)}\\
2\mu_2\chi^{(2)}&-2\mu_2a_S^{(2)}
}
= \bmat{t_{11}&t_{12}\\
t_{21}&t_{22}}.
\]
So
\begin{align*}
t_{11}&= -2\mu_1a_P^{(1)}-2\mu_1\chi^{(1)}a_P^{(2)}-2\mu_2a_P^{(2)} -2\mu_2\chi^{(2)}a_P^{(1)}
\\
&=-2\mu_1a_P^{(1)}-\htau^2a_P^{(2)}+2\mu_1a_P^{(2)}
-2\mu_2a_P^{(2)} -\htau^2a_P^{(1)}+2\mu_2a_P^{(1)}
\\
&= -\htau^2(a_P^{(1)}+a_P^{(2)})-2a_P^{(1)}(\mu_1-\mu_2)+2a_P^{(2)}(\mu_1-\mu_2)
\\
&= -\htau^2(a_P^{(1)}+a_P^{(2)})-2(a_P^{(1)}-a_P^{(2)})(\mu_1-\mu_2).
\end{align*}

Next
\begin{align*}
t_{21}&= -2\mu_1\chi^{(1)}+2\mu_1a_S^{(1)}a_P^{(2)}-2\mu_2a_P^{(2)}a_S^{(1)}+2\mu_2\chi^{(2)}
\\
&=-\htau^2+2\mu_1+2\mu_1a_S^{(1)}a_P^{(2)}
-2\mu_2a_P^{(2)}a_S^{(1)}+\htau^2-2\mu_2 \\
&= 2(\mu_1-\mu_2)+2a_S^{(1)}a_P^{(2)}(\mu_1-\mu_2)\\
&= (2+2a_S^{(1)}a_P^{(2)})(\mu_1-\mu_2).
\end{align*}
Next
\begin{align*}
t_{12}&= -2\mu_1a_P^{(1)}a_S^{(2)}+2\mu_1\chi^{(1)}-2\mu_2\chi^{(2)}+2\mu_2a_P^{(1)}a_S^{(2)}\\
&= -2\mu_1a_P^{(1)}a_S^{(2)}+\htau^2-2\mu_1-\htau^2+2\mu_2 +                                            2\mu_2a_P^{(1)}a_S^{(2)}\\
&= -2(\mu_1 - \mu_2)-2a_P^{(1)}a_S^{(2)}(\mu_1-\mu_2)\\
&=-(2+2a_P^{(1)}a_S^{(2)})(\mu_1-\mu_2).
\end{align*}
Then
\begin{align*}
t_{22}&= -2\mu_1\chi^{(1)}a_S^{(2)}-2\mu_1a_S^{(1)}-2\mu_2\chi^{(2)}a_S^{(1)}-2\mu_2a_S^{(2)}\\
&=-\htau^2a_S^{(2)}+2\mu_1a_S^{(2)}-2\mu_1a_S^{(1)}-\htau^2a_S^{(1)}+2\mu_2a_S^{(1)}
-2\mu_2a_S^{(2)}\\
&=-\htau^2(a_S^{(1)}+a_S^{(2)})+2a_S^{(2)}(\mu_1-\mu_2)-2a_S^{(1)}(\mu_1-\mu_2)\\
&=-\htau^2(a_S^{(1)}+a_S^{(2)})-2(a_S^{(1)}-a_S^{(2)})(\mu_1-\mu_2).
\end{align*}

It is worth noting here that $t_{21}$ and $t_{12}$ vanish when the parameters are equal, while the other two terms do not. This just means there is transmission of the $P$ and $S$ waves with no mode conversions, as to be expected when there are no interfaces.

So
\begin{align*}
\text{det} &= t_{11}t_{22}- t_{21}t_{12}\\
t_{11}t_{22}&= \htau^4(a_S^{(1)}+a_S^{(2)})(a_P^{(1)}+a_P^{(2)})\\
&\qquad +2\htau^2(a_S^{(1)}+a_S^{(2)})(a_P^{(1)}-a_P^{(2)})(\mu_1-\mu_2)
+2\htau(a_P^{(1)}+a_P^{(2)})(a_S^{(1)}-a_S^{(2)})(\mu_1-\mu_2)\\
&\qquad +4(a_S^{(1)}-a_S^{(2)})(a_P^{(1)}-a_P^{(2)})(\mu_1-\mu_2)^2\\
&= \htau^4(a_S^{(1)}+a_S^{(2)})(a_P^{(1)}+a_P^{(2)})\\
&\qquad +4\htau^2(a_P^{(1)}a_S^{(1)}-a_P^{(2)}a_S^{(2)})(\mu_1-\mu_2)\\
&\qquad +4(a_S^{(1)}-a_S^{(2)})(a_P^{(1)}-a_P^{(2)})(\mu_1-\mu_2)^2\\
&= \htau^4(a_S^{(1)}a_P^{(2)}+a_S^{(2)}a_P^{(1)})
+\textcolor{red}{\htau^4(a_S^{(1)}a_P^{(1)}+a_S^{(2)}a_P^{(2)})}\\
&\qquad +(\htau^2+2(\mu_1-\mu_2))^2a_P^{(1)}a_S^{(1)}+(\htau^2-2(\mu_1-\mu_2))^2a_P^{(2)}a_S^{(2)}\\
&\qquad \textcolor{red}{ -\htau^4 a_P^{(1)}a_S^{(1)}-4a_P^{(1)}a_S^{(1)}(\mu_1-\mu_2)^2
-\htau^4 a_P^{(2)}a_S^{(2)}-4a_P^{(2)}a_S^{(2)}(\mu_1-\mu_2)^2}\\
&\qquad +4(a_S^{(1)}-a_S^{(2)})(a_P^{(1)}-a_P^{(2)})(\mu_1-\mu_2)^2\\
&= \htau^4(a_S^{(1)}a_P^{(2)}+a_S^{(2)}a_P^{(1)})\\
&\qquad +(\htau^2+2(\mu_1-\mu_2))^2a_P^{(1)}a_S^{(1)}+(\htau^2-2(\mu_1-\mu_2))^2a_P^{(2)}a_S^{(2)}\\
&\qquad \textcolor{red}{ -4a_P^{(1)}a_S^{(1)}(\mu_1-\mu_2)^2
-4a_P^{(2)}a_S^{(2)}(\mu_1-\mu_2)^2}\\
&\qquad +4(a_S^{(1)}-a_S^{(2)})(a_P^{(1)}-a_P^{(2)})(\mu_1-\mu_2)^2\\
&= \htau^4(a_S^{(1)}a_P^{(2)}+a_S^{(2)}a_P^{(1)})\\
&\qquad +(\htau^2+2(\mu_1-\mu_2))^2a_P^{(1)}a_S^{(1)}+(\htau^2-2(\mu_1-\mu_2))^2a_P^{(2)}a_S^{(2)}\\
&\qquad \textcolor{red}{ -4a_P^{(1)}a_S^{(1)}(\mu_1-\mu_2)^2
-4a_P^{(2)}a_S^{(2)}(\mu_1-\mu_2)^2}\\
&\qquad \textcolor{red}{+4(a_P^{(1)}a_S^{(1)}+a_P^{(2)}a_S^{(2)})(\mu_1-\mu_2)^2}
\textcolor{blue}{-4(a_S^{(1)}a_P^{(2)}+a_S^{(2)}a_P^{(1)})(\mu_1-\mu_2)^2}\\
&= \htau^4(a_S^{(1)}a_P^{(2)}+a_S^{(2)}a_P^{(1)})\\
&\qquad +(\htau^2+2(\mu_1-\mu_2))^2a_P^{(1)}a_S^{(1)}+(\htau^2-2(\mu_1-\mu_2))^2a_P^{(2)}a_S^{(2)}\\
&\qquad \textcolor{blue}{-4(a_S^{(1)}a_P^{(2)}+a_S^{(2)}a_P^{(1)})(\mu_1-\mu_2)^2}
\end{align*}
Next, we have
\begin{align*}
-t_{21}t_{12} &= (2+2a_P^{(1)}a_S^{(2)})(2+2a_P^{(2)}a_S^{(1)})(\mu_1-\mu_2)^2\\
&=4(1+a_P^{(1)}a_P^{(2)}a_S^{(1)}a_S^{(2)})(\mu_1-\mu_2)^2
\textcolor{blue}{+4(a_S^{(1)}a_P^{(2)}+a_S^{(2)}a_P^{(1)})(\mu_1-\mu_2)^2}
\end{align*}
Thus, after cancelling the relevant terms, we obtain a nonzero determinant as long as $a^{(j)}_P$ and $a^{(j)}_S$ are not all complex:
\begin{align*}
\text{det}&= \htau^4(a_S^{(1)}a_P^{(2)}+a_S^{(2)}a_P^{(1)})\\
&\qquad +(\htau^2+2(\mu_1-\mu_2))^2a_P^{(1)}a_S^{(1)}+(\htau^2-2(\mu_1-\mu_2))^2a_P^{(2)}a_S^{(2)}\\
&\qquad +4(1+a_P^{(1)}a_P^{(2)}a_S^{(1)}a_S^{(2)})(\mu_1-\mu_2)^2.
\end{align*}
Next, notice that $t_{13},t_{23},t_{31},t_{32} = 0$. We may also calculate
\[
t_{33}= -\mu^{(2)}a_S^{(2)} - \mu^{(1)}a_S^{(1)} \neq 0
\]
away from glancing and this concludes our proof that $T$ is microlocally invertible in the relevant region.\\
\subsection*{Proof of Lemma \ref{l: jumps from reflection coeff}}

\label{appendix: proof of jumps from refl coeff}
We can now do the tedious computation required to prove Lemma \ref{l: jumps from reflection coeff}, which states that one may recover the infinitesimal jumps in wave speeds from the reflection coefficients.
\newline

\begin{proof}[Proof of Lemma \ref{l: jumps from reflection coeff}]
We would like to compute $R= Q_{21}T$ as well, or at the least check that it is invertible. As before, it suffices to check $[((S^{(1)})^TKE^{-1})(ES^{(2)})]_{21}$ is invertible. Then the first $2 \times 2$ minor of this matrix is given by the multiplication of

\[
\bmat{2\mu_1a_P^{(1)}&2\mu_1 \chi^{(1)}&1&a_P^{(1)} \\
-2\mu_1\chi^{(1)}&2\mu_1a_S^{(1)}&-a_S^{(1)}&1
}
\bmat{
1&a_S^{(2)}\\
-a_P^{(2)}&1\\
-2\mu_2a_P^{(2)}&-2\mu_2\chi^{(2)}\\
2\mu_2\chi^{(2)}&-2\mu_2a_S^{(2)}
}
= \bmat{z_{11}&z_{12}\\
z_{21}&z_{22}}.
\]

First, we have
\begin{align*}
z_{11}&=2\mu_1a_P^{(1)}-2\mu_1\chi^{(1)}a_P^{(2)}-2\mu_2a_P^{(2)}+2\mu_2 a_P^{(1)}\chi^{(2)} \\
&=
2\mu_1a_P^{(1)}-a_P^{(2)}(\hat \tau^2-2\mu_1)-2\mu_2a_P^{(2)}+a_P^{(1)}(\hat \tau^2-2\mu_2)\\
&=
\hat \tau^2(a_P^{(1)}-a_P^{(2)})+2a_P^{(1)}(\mu_1-\mu_2)+2a_P^{(2)}(\mu_1-\mu_2)\\
&= \hat \tau^2(a_P^{(1)}-a_P^{(2)})+2(a_P^{(1)}+a_P^{(2)})(\mu_1-\mu_2).
\end{align*}

Next, we have
\begin{align*}
z_{21}&= -2\mu_1\chi^{(1)}-2\mu_1a_P^{(2)}a_S^{(1)}+2\mu_2a_P^{(2)}a_S^{(1)} + 2\mu_2 \chi^{(2)}\\
&=
-(\hat \tau^2-2\mu_1) -2a_P^{(2)}a_S^{(1)}(\mu_1-\mu_2) + (\hat \tau^2-2\mu_2)\\
&=
2(\mu_1-\mu_2)-2a_P^{(2)}a_S^{(1)}(\mu_1-\mu_2)
\\
&=
(2-2a_P^{(2)}a_S^{(1)})(\mu_1-\mu_2).
\end{align*}
Continuing,
\begin{align*}
z_{12}&= 2\mu_1a_P^{(1)}a_S^{(2)}+2\mu_1 \chi^{(1)}-2\mu_2\chi^{(2)}-2\mu_2a_P^{(1)}a_S^{(2)}\\
&=2a_P^{(1)}a_S^{(2)}(\mu_1-\mu_2)-2\mu_1+2\mu_2\\
&=(-2+2a_P^{(1)}a_S^{(2)})(\mu_1-\mu_2).
\end{align*}
Lastly,
\begin{align*}
z_{22}&= -2\mu_1a_S^{(2)}\chi^{(1)}+2\mu_1a_S^{(1)}+2\mu_2 a_S^{(1)}\chi^{(2)}-2\mu_2a_S^{(2)}\\
&= -a_S^{(2)}(\hat \tau^2-2\mu_1)+2\mu_1a_S^{(1)}+a_S^{(1)}(\hat \tau^2-2\mu_2)-2\mu_2 a_S^{(2)}\\
&=\hat \tau^2(a_S^{(1)}-a_S^{(2)}) +2a_S^{(2)}(\mu_1-\mu_2)+2a_S^{(1)}(\mu_1-\mu_2)\\
&=\hat \tau^2(a_S^{(1)}-a_S^{(2)}) +2(a_S^{(1)}+a_S^{(2)})(\mu_1-\mu_2).
\end{align*}
We also have
\[
z_{33} =  -\mu^{(2)}a_S^{(2)} + \mu^{(1)}a_S^{(1)}
\]

It will be convenient to denote $R=\bmat{r_{11}&r_{12}&r_{13}\\
r_{21}&r_{22} & r_{23}\\r_{31}&r_{32}&r_{33}}$ the individual entries.
Next, notice that $r_{13},r_{23},r_{31},r_{32} = 0$ since the corresponding entries for $T$ and $Q_{21}$ are as well. Using the calculation for $T$, we may calculate
\[
r_{33}= \frac{\mu^{(1)}a_S^{(1)} - \mu^{(2)}a_S^{(2)}}{ \mu^{(1)}a_S^{(1)} + \mu^{(2)}a_S^{(2)}}.
\]
We may then compute $(r_{33}-1)/(r_{33}+1) = (\mu^{(2)}a_S^{(2)})/(\mu^{(1)}a_S^{(1)})$, so
since $\mu^{(1)}$ and $a_S^{(1)}$ is already determined, we recover $\mu^{(2)}a_S^{(2)} = \sqrt{\mu^{(2)}|\xi'|^2 - \tau^2}$. Since the tangential momenta $\xi',\tau$ are already determined, we recover $\mu^{(2)}(x_0)$ and thereby $a_S^{(2)}(x_0,\tau_0,\xi_0)$.

All we have left to determine is $a_P^{(2)}$ which would give us $\lambda^{(2)}(x_0)$. For this, we use the first $2\times 2$ minor of $T$ and $Q_{21}$, and after a tedious computation, since everything is known except $\lambda^{(2)}$, we get the recovery by similar arguments as above.

We assumed throughout these calculations that $|\xi'|$ lies away from zero. However, at $0$, the calculations are much simpler and follow the same arguments.
\end{proof}

\section{Blind Scattering Control}\label{s: blind scatt control}
In this section, we obtain results on how much scattering may be controlled and eliminated using only the outside measurement operator. This will determine how much of the main results of \cite{CHKUControl} may actually be extended to the elastic setting. We will show that in general, without more information, one does not have control in regions outside of where one has elastic unique continuation theorems. We emphasize that the results here are of a very different nature than in the main paper. In the main text, we layer strip and so we assume knowledge of the wave speeds in a certain portion of the medium. Here, we have no \emph{a priori} knowledge of the Lam\'{e} parameters inside the medium, and we aim to determine what type of probing may be done into the medium using time reversal methods.

\subsubsection*{Subsets}
We will need certain subsets of $\Theta$ determined by the boundary distance function.
For each $t$, define the open sets
\begin{align*}
\Theta_t &= \set*{x\in \Upsilon}{d^*_{\Theta}(x)>t},&
\Theta_t^\star&=\set*{x \in \Upsilon}{d^*_{\Theta}(x)<t}.
\end{align*}

There is some $M$ for which $\Theta_t$ is Lipschitz for each $t\leq M$, and we assume for this section that $\Theta_t$ is Lipschitz for the parameters $t$ we consider.
We use the superscript $\star$ to indicate sets and function spaces lying outside, rather than inside, some region. Sets $\Omega_t^{\ },\, \Omega_t^\star$ may be defined in the same way. We also have $\Theta^*=\Theta_0^*$.

\subsubsection*{Projections Inside and Outside $\Theta_t$}

The final ingredients needed for the iterative scheme are restrictions of Cauchy data inside and outside $\Theta$. While a hard cutoff is natural, it is not a bounded operator in energy space: a jump at $\p \Theta$ will have infinite energy. The natural replacement are Hilbert space projections. More generally, we consider projections inside and outside $\Theta$.
Let $\pi^\star$ be the orthogonal projections of $\mathbf C$ onto $\mathbf H^\star$; let $\bar{\pi} = \Id - \pi^\star$. For $u \in \mathbf C$, note that $\pi^\star u = u$ on $\Theta$ while outside $\Theta$ it is smooth since it solves an elliptic equation with smooth Lam\'{e} parameters.

\subsubsection*{Almost direct transmission}
Suppose we have Cauchy data $h_0\in\mathbf H$. We can probe $\Omega$ with $h_0$ and observe $Rh_0$ outside $\Omega$. In particular, the reflected data $\pi^\star R$ can be measured, and from these data, we would like to procure information about $c_{\PS}$ inside $\Omega$. However, multiple scattering caused by waves travelling into and out of $\Omega$ make $\pi^\star Rh_0$ difficult to interpret.

In this section, we construct a control in $\mathbf H^\star$ that eliminates multiple scattering in the wave field of $h_0$ up to a certain depth dependant on $T$ inside $\Theta$. More specifically, consider the \emph{almost direct transmission} of $h_0$:

\begin{definition}\label{def: almost dt}
	The \emph{almost direct transmission}, denoted $h_{\text{DT}}$, of $h_0\in\mathbf H$ at time $T$ is $\bar \pi_T R_Th_0$, where $\bar \pi_T$ is defined the same way as $\bar \pi$ but replacing $\Theta$ with $\Theta_T$. 
\end{definition}
We note since it will end up being impossible to actually recover the direct transmission using outside data, we do not give a more natural definition of $\pi_T$, but merely use a convenient one to illustrate a particular point in the next lemma.

By finite speed of propagation, $h\DT$ is equal to $R_Th_0$ inside $\Theta_T$; outside $\Theta_T$, its first component satisfies the elliptic Laplace equation, while the second component is extended by zero.

\subsubsection*{Scattering Control Series}
We want Cauchy data of the form
\[
h_{\infty} = K\tail + h_0
\]
with $K\tail$ supported outside of $\Theta$ such that $K\tail$ cancels multiples produced by $h_0$ and allows us to recover the directly transmitted constituent of $R_T h_{\infty}$ (even though we will show this to be impossible in the elastic setting without more information). More specifically, $K\tail$ will be in $\mathbf H^\star$ and it is the control discussed in the previous subsection.

Next, let us write down the \emph{scattering control equation} first introduced in \cite{CHKUControl}. With $h_0$ and $h_\infty$ as above, we say $h_\infty$ is a solution to the scattering control equation if the following holds
\begin{equation}
(I-\pi^\star R \pi^\star R)h_{\infty} = h_0.
\end{equation}

Using only finite speed of propagation, we can obtain a necessary condition for $K\tail$ in order to control multiple scattering.

\begin{lemma}\label{l: SC in necessary for K_tail}
Let $h_0 \in \mathbf{H}$ and $T \in (0,\frac{1}{2}\text{diam}_p\Theta)$. Then isolating the deepest part of the wavefield associated to $h_0$ implies $K\tail \in \mathbf H^\star$ satisfies the scattering control equation:
\begin{equation}\label{eq: scattering control}
R_{-T}\bar{\pi}R_{2T}h_{\infty} = h_{DT} \Longrightarrow (I-\pi^\star R \pi^\star R)h_{\infty} = h_0.
\end{equation}
\end{lemma}

\proof
The proof is actually identical to the one in \cite{CHKUControl}, but we provide the argument here just for completeness. Assume $R_{-T}\bar{\pi}R_{2T}h_{\infty} = h_{DT}.$
Observe that due to the disjoint support of $K\tail$ and $h_0$ one has
$\bar \pi h_{\infty} = h_0$. Hence
\begin{align*}
(I-\pi^\star R \pi^\star R)h_{\infty}
&= (I - \pi^\star R (I-\bar \pi)R)h_{\infty}\\
&= h_{\infty}-\pi^\star R^2 h_{\infty} + \pi^\star R \bar \pi Rh_{\infty}\\
&= h_{\infty}-\pi^\star R_{-2T}R_{2T}h_{\infty} + \pi^\star \nu R_{2T}\nu \bar \pi R_{2T}h_{\infty}\\
&=h_{\infty}-\pi^\star h_{\infty} + \pi^\star R_{-2T}(R_T h_{DT})\\
&= h_0 + \pi^\star R_{-T} h_{DT}.
\end{align*}
Hence, we merely need the second term on the right to vanish. Notice $h_{DT}$ is supported inside $\Theta_T$ so by finite speed of propagation, $R_{-T}h_{DT}$ has support inside $\Theta$ since $T$ is based on the fast $P$-wavespeed. Hence, $\pi^\star R_{-T}h_{DT} = 0$. $\Box$

This was the ``easy'' direction. The converse of the above lemma is far from clear and will turn out false. Indeed, the converse would say that if $K\tail$ satisfies a certain equation outside $\Theta$, then we can retrieve information about $u(t) = Fh_{\infty}$ inside $\Theta$ or at least inside $\Theta_T^\star$, revealing the elimination of certain scattering within $\Theta^\star$ by time $T$. This would require a unique continuation theorem, but for the elastic equation, a unique continuation theorem can only give information within the $s$-domain of influence of $h_{\infty}$. Such a theorem would say absolutely nothing inside the remaining portion of the $P$-domain of influence of $h_{\infty}$. Thus, we deem the task of obtaining such a result in this exact framework impossible with known elastic unique continuation results. Nevertheless, we still have a lemma that will help characterize $R_T h_{\infty}$ even in the elastic setting, which also provides a general description of
 which multiply scattered wave constituents may be eliminated.

\begin{lemma}
With the same setup as in the previous lemma, assume
\[(I-\pi^\star R \pi^\star R)h_{\infty} = h_0.\]
Then in the energy inner product, one has
\[
R_{-T}\bar \pi R_{2T}h_\infty \perp \text{Im}(R_T \pi^\star)+\text{Im}( R_{-T} \pi^\star).
\]
Denote $\mathcal C:= \overline{\text{Im}(R_T \pi^\star)+\text{Im}(\nu R_T \pi^\star)}$
the controllable subspace, $\Pi_\mathcal C$ its orthogonal projection with respect to the energy inner product, and $\Pi_{\mathcal C^\perp}$ the orthogonal projection onto its orthogonal complement (the uncontrollable subspace).
Then
\begin{equation}\label{eq: exact proj scat control}
R_{-T}\bar \pi R_{2T}h_\infty = \Pi_{\mathcal C^\perp}R_T h_0.
\end{equation}
Conversely, if $h_{\infty}$ satisfies (\ref{eq: exact proj scat control}), then $h_\infty$ satisfies the scattering control equation.
\end{lemma}

\begin{rem}
The above lemma shows that the scattering control equation produces a tail that eliminates any constituent of the wavefield at time $t=T$ that may be recreated by $R_T \pi^\star$. That is, the wavefield produced by an internal source can be split (microlocally) into a directly transmitted component (experiencing no reflections), singly-reflected components, and multiply-reflected components. Any such multiples that may be artificially created using Cauchy data supported outside $\Theta$ (which is precisely the image of $R_T\pi^\star$) are precisely what get eliminated. Of course, Cauchy data supported in $\Theta^\star$ can never recreate the directly transmitted portion of $h_0$ simply due to finite speed of propagation; hence the directly transmitted piece may never be eliminated, as was obvious \emph{a priori}. Likewise, the single reflections cannot be recreated by Cauchy data in $\mathbf H^\star$ and so those do not get eliminated either.
\end{rem}

\begin{proof}
The proof will take advantage of $R_T$ being a unitary operator with respect to the energy inner product. Next, note that the scattering control equation is equivalent to $\pi^\star R_{-2T}\bar \pi R_{2T}h_\infty = 0$ as shown in \cite{CHKUControl}.
Take any Cauchy data $k \in \mathbf C$. One has
\begin{align*}
\langle R_{-T}\bar \pi R_{2T}h_\infty , R_T\pi^\star k \rangle
&= \langle R_{-2T}\bar \pi R_{2T}h_\infty , \pi^\star k \rangle
\\
&= \langle \pi^\star R_{-2T}\bar \pi R_{2T}h_\infty , \pi^\star k \rangle
+ \langle \bar \pi R_{-2T}\bar \pi R_{2T}h_\infty , \pi^\star k \rangle
\\
&= 0,
\end{align*}
where the first quantity is $0$ by the scattering control equation $(\pi^\star R_{-2T}\bar \pi R_{2T}h_\infty = 0)$ and the second quantity is $0$ because $\pi^\star$ and $\bar\pi$ are orthogonal projectors. For $\nu R_T \pi^\star$, one instead has
\begin{align*}
\langle R_{-T}\bar \pi R_{2T}h_\infty , \nu R_T\pi^\star k \rangle
&=
\langle R_{-T}\nu R_{-T}\bar \pi R_{2T}h_\infty , \pi^\star k \rangle
\\
&=
\langle \nu R_TR_{-T}\bar \pi R_{2T}h_\infty , \pi^\star k \rangle
=
\langle \nu \bar \pi R_{2T}h_\infty , \pi^\star k \rangle = 0.
\end{align*}

Let $\Pi_\mathcal C$ and $\Pi_{\mathcal C^\perp}$ be the orthogonal projections onto $\mathcal C:= \overline{\text{Im}(R_T \pi^\star)+\text{Im}(\nu R_T \pi^\star)}$ and its orthogonal complement, respectively..
By continuity, the above proof shows \[R_{-T}\bar \pi R_{2T}h_\infty = \Pi_{\mathcal C^\perp}R_{-T}\bar \pi R_{2T}h_\infty.\]
On the other hand, one has
\begin{align*}
R_{-T}\bar \pi R_{2T}h_\infty
&=- R_{-T} \pi^\star R_{2T}h_\infty + R_Th_{\infty}
\\
&= -\nu R_T \pi^\star R h_\infty + R_T \pi^\star K\tail
+ R_Th_0
\end{align*}
Notice that $\Pi_{\mathcal C^\perp}$ applied to the first two terms vanishes.
Thus, we have
\[
R_{-T}\bar \pi R_{2T}h_\infty = \Pi_{\mathcal C^\perp}R_Th_0
\]
as desired.

For the last part, suppose that $R_{-T}\bar \pi R_{2T}h_\infty = \Pi_{\mathcal C^\perp}R_T h_0$. Then $R_{-T}\bar \pi R_{2T}h_\infty \perp \text{Im}(R_T\pi^\star)$. Since $R_T$ is unitary in the energy inner product, this implies
\[R_{-2T}\bar \pi R_{2T}h_\infty \perp \text{Im}(\pi^\star).\]
Thus, $R_{-2T}\bar \pi R_{2T}h_\infty$ paired with any Cauchy data supported in $\Theta^\star$ is $0$. Thus, it must be that
\begin{align*}
\pi^\star R_{-2T}\bar \pi R_{2T}h_\infty &=0\\
\Leftrightarrow (I-\pi^\star R \pi^\star R)h_\infty &= 0.
\end{align*}
\end{proof}

\subsection*{Elastic Unique Continuation}
We have the unique continuation principle for smooth parameters $\lambda,\,\mu$, described in \cite{KNTAT} and proved by Belishev in \cite{BelLameType}.
\begin{theorem}
Suppose that $(\p_t^2 - L)u=0$ in $\RR \times \RR^3$ and $u = 0$ in a neighborhood of $(t_0,t_1) \times \{x_0\}$. Then, $u=0$ for all $(t,x)$ inside the double cone
\[
\left\{
(t,x): d_S(x,x_0) +
\left| t- \frac{t_0+t_1}{2}\right| \leq
\frac{t_1 - t_0}{2}
\right\}.
\]
\end{theorem}

Just as in the scalar wave setting, it is believed that the unique continuation principle holds for Lam\'{e} parameters with certain conormal singularities as well. This can be proven using the method of Eller-Isakov-Nakamura-Tataru for isotropic elasticity \cite{EINT02} and the modifications for a discontinuous sound speed as developed in reference \cite{SU-TATBrain}.
Thus, for the purposes of this work, we merely assume that the Lam\'{e} parameters satisfy the above unique continuation principle, without requiring smoothness.
\\

\noindent{\bf Assumption.} $\lambda,\mu$ satisfy the above unique continuation principle.
\\

We then have the following corollary whose proof is identical to \cite[Lemma 3.5]{KNTAT}:
\begin{cor}
Assume $u$ is a solution to the elastic equation in free space, $T>0$ and $u= 0$ on $[-T,T] \times \Omega^\star$. Then by the unique continuation principle, $u(t,x) = 0$ for all $(t,x)$ such that $d_S(x,\p \Omega) + |t| \leq T$.
\end{cor}

We may now prove a corollary that essentially states that
\[
R_{-T}\bar \pi R_{2T}h_\infty\Big|_{\{d_S(x,\p \Theta) <T\}} \approx 0,
\]
that is, all waves in the time $T$ $s$-domain of influence of the initial data is controlled.
\begin{cor}
Assume $h_0,h_\infty$ and $T$ are as in Lemma \ref{l: SC in necessary for K_tail} and $h_\infty$ satisfies the scattering control equation.
For the elastic equation, all waves inside the slow domain of influence $\{d_s(x,\p \Theta) <T\}$ are controlled:
\begin{align*}
j_2R_{-T}\bar \pi R_{2T}h_\infty|_{\{d_S(x,\p \Theta) <T\}} &= 0,\\ Lj_1R_{-T}\bar \pi R_{2T}h_\infty|_{\{d_S(x,\p \Theta) <T\}} &= 0,
\end{align*}
where $j_{1/2}$ are projections to the first or second components of the Cauchy data respectively. That is, $R_{-T}\bar \pi R_{2T}h_\infty$ is stationary on the slow domain of influence.
\end{cor}

\begin{rem}
We note that the elliptic equation is the stationary elastic equation, where the parameters do have singularities. If one wanted to avoid this, then instead of using projections $\pi^ \star$, one could instead use smooth cutoffs, and obtain actual vanishing for the first component at the cost of an epsilon-type loss in the $s$-domain of influence set. Since the second part of the paper is more fundamental, we do not give more details on this.
\end{rem}

\proof
Indeed, suppose $h_{\infty}$ satisfies
\[
(I-\pi^\star R \pi^\star R)h_{\infty} = h_0.
\]
Let $v = FR_{-2T}\bar \pi R_{2T}h_{\infty}$, and let $\mathbf v(t)=(v(t),\partial_t v(t))\in\mathbf C$. Then $\mathbf{v}(2T) = \bar \pi R_{2T}h_{\infty}$ implies the second component of $\mathbf{v}(2T)$ is zero on $\Theta^\star$ while the first component is smooth. That is, $\mathbf v(2T)$ is stationary on $\Theta^\star$. Also,
\[
\mathbf{v}(0)= R_{-2T}\bar \pi R_{2T}h_{\infty}
= h_{\infty}-R_{-2T} \pi^\star R_{2T}h_{\infty}
= (I-R\pi^\star R)h_{\infty}
\]
which is stationary on $\Theta^\star$ using our assumption. Thus, as in \cite[proof of theorem 2.2]{CHKUControl}, if we replace $v$ by $w= \p_t w$, then $\mathbf w = 0$ on $\Theta^\star$ and $w$ also satisfies the elastic equation (it is a distributional solution). By unique continuation, $\mathbf w$ vanishes for $(t,x)$ such that $d_S(x,\p \Theta) + |T-t| \leq T$. In particular, the second component of $\mathbf v(T) = 0$ in $ \{x \in \Theta: d_S(x,\p \Theta)\leq T\}$ corresponding to the slow domain of influence, while the first component solves an elliptic equation by definition of $w$.
$\Box$
\\

It is now clear that understanding $\text{Im}(R_T\pi^\star)$ tells us exactly which scattered waves get eliminated through scattering control. Notice that the acoustic analog of the above corollary says that one has full control over the time $T$ domain of influence of $h_0$ as was already proven in \cite{CHKUControl}. Since we do not have full unique continuation in the elastic setting, we instead rely on microlocal analysis to solve a geometric problem.
\\

\section{Proof of proposition \ref{prop: internal source construction}}\label{s: proving internal source construction}
Before proving this proposition, it will be useful to have a notion of the microlocal counterpart to the almost direct transmission (see Definition \ref{def: almost dt}) from the previous appendix.

\subsubsection*{Microlocal almost direct transmission}

As in \cite[section 3]{CHKUControl}, we are interested in isolating the \emph{microlocal almost direct transmission} since this will be the main tool necessary to prove the main theorem in the presence of multiple scattering. Intuitively, it is the microlocal restriction of the solution at time $T$ (say) to singularities in $T^*\Theta$ whose $P$-distance from the surface $\p T^*\Theta$ is at least $T$. Formally, first let $(T^*M)_t$ be the set of covectors of depth greater than $t$ in a manifold $M$:
\[
(T^*M)_t=\{ \xi \in T^*M| d^*_{T^*M}(\xi) >t \}
\]
where $d^*_{T^*M}(\xi)$ is defined as in \cite[section 3]{CHKUControl} using $d_P$.

 Instead, we consider smooth cutoffs, and for choose nested open sets $\Theta',\Theta''$ between $\Omega$ and $\Theta$:
\[
\Omega \subset \Theta' \subset \overline{\Theta'} \subset \Theta'' \subset \overline{\Theta''} \subset \Theta.
\]

A \emph{microlocal almost direct transmission} of $h_0$ at time $T$ is a distribution $h\MDT$ satisfying
\[
h\MDT \equiv R_T h_0 \qquad \text{on }(T^*\Theta)_T \qquad \qquad \text{WF}(h\MDT) \subset \overline{(T^*\Theta'')_T}.
\]

The key to prove this proposition is to isolate $h\MDT$. For this, we need access outside of $\Omega$ to all scattered rays related to $h\MDT$ (\emph{bad} bicharacteristics defined below), which certainly includes all possible $S$-rays as well.

\begin{definition}

\begin{enumerate}
\item[(a)]
 Let $\gamma :[0,t^*] \to T^*\RR^3$ be a purely transmitted, broken, $P$ or $S$ ray that starts outside $\Omega$ and hits $k$ interfaces at discrete times $t_1,\dots,t_k >0$ say. At each such point, $\gamma$ will have one or two reflecting branches and one possible mode converted branch according to Snell's law.
Such branches that are reflections are called \emph{bad reflecting rays} associated to $\gamma$. A transmitted branch that is a mode conversion associated to $\gamma$ will be called a \emph{bad transmission}. A wave associated with such a ray will be called a bad wave and we may refer to either branch as a bad ray or bad branch. For the proof of the proposition, such bad reflecting waves are precisely the ones that will create waves (upon their next interaction with an interface) that need to be eliminated. We must also ensure that these bad mode-conversion transmissions are eliminated as well.

\item[(b)]
Bicharacteristics $\gamma_1,\gamma_2$ are \emph{connected} if their concatenation $\gamma_1 \cup \gamma_2$ is a broken bicharacteristic.
Note that mode conversions are allowed (e.g., a $P$ ray may be connected to an $S$ ray) if their tangential momenta match.
A bicharacteristic $\gamma_1$ terminating at an interface may have one or two (totally reflected with $\PS$ mode conversions), or two, three, or four (reflected and transmitted) connecting bicharacteristics there. If $\gamma_1$ has a transmitted bicharacteristic, there exists an \emph{opposite} bicharacteristic $\gamma$ sharing $\gamma_1$'s connecting bicharacteristics. There can be up to two opposite bicharacteristics (one for $P$ and one for $S$). Note that $\gamma_1$ or $\gamma_2$ (or both) may be glancing at an interface (see \cite{CHKUControl} for definitions). If it is not, we say that it is \emph{non-glancing}.
\\
\item[(c)] Fix a large time $T_s >0$ (see below). A bicharacteristic $\gamma:(t_-,t_+) \to T^*(\RR^3\setminus \Gamma)$ is ($\pm$)-\emph{escapable} if either:\\
    \begin{enumerate}
    \item[i.] it has \emph{escaped}: $\gamma$ is defined at $T\pm T_s$
    and $\gamma(T\pm T_s) \notin T^*\Theta$,
    \\
    \item[$\bullet$]
    or recursively, after only finitely many recursions, either
    \\
    \item[ii.] all of its connecting bicharacteristics at $t_{\pm}$ are ($\pm$)-escapable and are non-glancing at the interface;
        \\
    \item[iii.]
    All of its reflecting bicharacteristics (both $P$ and $S$ if they exist) are $(\pm)$ escapable and non-glancing
    and both $P$,$S$ opposite bicharacteristics are $(\mp)$ escapable and non-glancing.
\\
    \end{enumerate}
\end{enumerate}

Let us explain the choice of $T_s$ further. The idea is that we want a large enough time so that any returning bicharacteristic, even a concatenation of pure slow rays, will return to $\Theta^\star$ by time $T_s$. This will ensure that there is enough time for all wave constituents of a particular $u_h$ associated to returning bicharacteristics eventually return to $\Theta^\star$ by time $T_s$. This avoids the problems encountered in the lack of control in the previous appendix. We also note that $T_s$ is not used to discriminate between certain rays and we can even allow it to vary for different rays since we are allowed infinite time to take exterior measurements of the wavefield. Rather, it just makes the notation less cumbersome to have a fixed time beyond which other rays are irrelevant and will not affect the construction of the ``tail''

The idea is that $(+)$ escapable singularities are ones we do not worry about since they escape and do not enter the directly transmitted region. Once there are connecting rays that are not $(+)$ escapable, then those need to be eliminated. Thus, the third condition guarantees that corresponding to these non-escapable rays, there are corresponding \emph{opposite} $(-)$ escapable rays that we use to send in waves to eliminate waves associated to rays that do not escape. We will also refer to such rays or waves that are non-escapable and enter the directly transmitted region as \emph{bad} rays, resp. \emph{bad} waves. The previous definition of bad rays are exactly the ones that create (through geometric optics concatenation) the bad non-escapable rays just described.

In the final case, if the ($\pm$)-escapable connecting bicharacteristic is a reflection, then we require that there are both $P$ and $S$ opposite bicharacteristics that escape. They must be there so that we can construct an incoming wave parametrix, singular on such opposite rays, that eliminates a particular scattered wave. They must also escape so that we obtain all the necessary information not just from the fast moving $P$-waves, but the $S$-waves as well. This is because we must construct $S$ waves in addition to $P$ waves for the tail, even if one is merely trying to eliminate a single $P$ wave.
\end{definition}

\begin{rem}
The definition of escapability merely ensure that we obtain access in our measurement region of all backscattered rays caused by the direct transmission. In addition, the recursive definition and the notion of \emph{opposite} rays ensures that for any new bad scattered waves created by an appropriate tail, one will be able to eliminate them as well if they enter the direct transmission's domain of influence. The conditions should be compared to the linear problem of obtaining an observability inequality with partial data. There as well, one needs access to all the relevant rays in the measurement region.
\end{rem}

\begin{rem}
Note that the definition of $(\pm)$-escapable rays can be made more general to deal with more general geometries than our convex foliation case. Since we only need to eliminate waves associated to rays that are not escapable, this only requires having the associated number of opposite escapable rays. Thus, we need the number of opposite $(\mp)$-escapable rays to equal the number of not $(\pm)$-escapable rays that are reflecting or transmitting.
\end{rem}

Let us recall $\mathcal S \subset T^*\Omega$ as the set of $\xi$ such that every bad bicharacteristic through $\xi$ is \emph{(+)-escapable}. Note that the definition of escapable ensures that the mode conversion are non-glancing as well. For example, if a purely transmitted $P$ bicharacteristic starts outside $\Omega$ and passes through $\xi$, then all the transmitted $S$, mode-converted connecting rays are non-glancing. An analog holds for a purely transmitted $S$ bicharacteristic.

We will also need the directly transmitted component of the forward elastic wave propagator $R_T^+$.
\begin{definition}[The microlocal direct transmission]\label{def: microlocal direct trans}
Fix $(x,\xi) \in \mathcal S$ and let $\gamma_\PS$ be a purely transmitted $P$ or $S$ ray, starting outside $\Omega$ at $t=0$ and passing through $(x,\xi)$ at some time $t= T$. Assume $\gamma_\PS$ intersects $\Gamma$ exactly $k$ times. Define the \emph{directly transmitted wave constituent} FIO
\begin{equation}
\mathbf{DT}^+_{k,\PS} = \begin{cases}
r_T \Pi_\PS J\CtoS,  & k=0,\\
r_T \Pi_\PS J\BtoS \iota M_T(J\BtoB \Pi_\PS \iota M_T)^{k-1} J_{\mathbf S \to \p} \Pi_\PS J\CtoS,   & k>0
\end{cases}
\end{equation}
By construction, if $h_0$ is a distribution of Cauchy data with $\text{WF}(h_0)=\RR_+ \gamma_\PS(0)$, then $\mathbf{DT}^+_{k,\PS}h_0$ will be a distribution whose wavefront set is equal to $(x,\RR_+\xi)$. Moreover, it will be a polarized $\PS$ wave.
\end{definition}

(Motivation of the proof) In the proof we will define FIO's $\Xi_{\pm}$ that will produce the correct tails (controls outside $\Theta$) associated to $\pm$-escapable bicharacteristics to eliminate certain \emph{bad} waves.
Suppose $y= y(t,x)$ is a microlocal wave constituent ($P$ or $S$ wave) of some wavefield $u_h$ associated to a returning bicharacteristic $\gamma$. It may be represented locally by an FIO analogous to (\ref{e: Cauchy propagator}). Let $\Gamma$ be the next interface $y$ hits. Let $y' = \rho_\Gamma y$ where $\rho_\Gamma$ is the restriction operator. Now, $y$ will produce $P$ and $S$ reflected waves determined by $M_R y'$. If these waves will eventually interfere with the direct transmission, they must be eliminated (these are waves associated to non-escapable rays connected to $\gamma$). Note that if both reflected and transmitted waves from $y$ don't cause waves that interfere with the directly transmitted region, this is case $(ii)$ in the definition and nothing would need to be done. Otherwise, let $\gamma^{\PS}$ be the two opposite rays from $\gamma$, thought of as lying inside $T^*(\RR \times \Theta)$. Thus, we must send in a waves associated to the $\gamma^{\PS}$ that will eliminate $M_R y'$. That is, we want to produce an incoming wave $z$ on the side of $\Gamma_-$, such that $z':=\rho_{\Gamma}z = -M_T^{-1}M_Ry'$. The outgoing waves on the side of $\Gamma_-$ are then $M_T y'$ and $-M_RM_T^{-1}M_Ry'$ while our construction ensures there are no outgoing waves on the side of $\Gamma_+$ that will pollute the directly transmitted region. Notice that the outgoing waves $M_T y'$ and $-M_RM_T^{-1}M_Ry'$ might produce more bad waves later on (propagating forwards $(+)$) so $\Xi_+$ will need to take care of these. The new incoming wave $-M_T^{-1}M_Ry'$ might also produce bad waves in the backward direction $(-)$ so $\Xi_-$ will have to take care of those. Thus, $\Xi_{+}$ needs to be defined recursively at this step as
\[
-\Xi_-M_T^{-1}M_Ry' + \Xi_+(M_T-M_RM_T^{-1}M_R)y'.
\]
Similarly, if the transmitted part $M_Ty'$ would cause bad waves, then the definition would be
\[
-\Xi_-M_R^{-1}M_Ty' + \Xi_+(M_R-M_TM_R^{-1}M_T)y'.
\]
In this case, we take care of all possible branches (and multi-branches) that may be associated with $\gamma$ and produce bad waves.

Lastly, before beginning the proof, we cannot construct our usual parametrices near glancing rays. Fortunately, the convex foliation condition will guarantee that ``most'' (in a sense to be described soon) broken bicharacteristics (that travel for a fixed finite time $T$ say) will not glance at an interface.

\begin{lemma}\label{l: glancing rays are small}
Let $\mathcal G = \mathcal G_T \subset T^*\Theta$ be the set of $(y,\eta) \in S^*\Theta$ such that a broken bicharacteristic of length $T$ contains a glancing point. $\mathcal G$ is a manifold and under the convex foliation assumption, and has dimension at most $2n-2$. Thus, the set of covectors $(y,\eta) \in S^*\Theta$ where all broken rays of length $T$ passing through $(y,\eta)$ never glance is dense in $S^*\Theta$.
\end{lemma}

We also have a series of lemmas to demonstrate that under the convex foliation condition, we have enough covectors lying in $\mathcal S$.

\begin{lemma}\label{l: mono increase}
Let $\gamma$ be a transmitted geodesic with respect to some wave speed $c.$ Then $\rho \circ \gamma$ either monotonically decreases, strictly monotonically increases, or strictly decreases then strictly increases.
\end{lemma}

\proof Suppose, on the contrary, that $\rho \circ \gamma$ is nondecreasing on $[a,b]$ then nonincreasing on $[b,d]$ for some $a<b<d$. Let $\tau = \rho(b)$. If $c$ is smooth near $\gamma(b)$ then there is a neighborhood $(a',d') \subset [a,d]$ of $b$ such that $\rho(\gamma(a')) = \rho(\gamma(d')) = \tau' \geq \tau$. Then $\gamma|_{[a',d']}$ is a geodesic between points on $\Sigma_{\tau'}$ outside of $\Omega_{\tau'}$, contradicting the strict convexity of $\p \Omega_{\tau'}$. Conversely, if $c$ is discontinuous at $\gamma(b)$, then $\gamma((a,b))$ and $\gamma((b,d))$ are on opposite sides of $\Gamma$, which is locally given by $\Sigma_\tau$, by the definition of a transmitted geodesic. This is a contradiction as well. $\Box$

The next lemma states that upward-travelling geodesics are not trapped.

\begin{lemma}\label{l: purely transmitted rays are dense}
The set of $(x,\xi) \in T^*_-\Omega$ for which there exists a purely transmitted geodesic $\gamma: [a,b] \to \overline\Omega$ with $\gamma'(0)^\flat = (x,\xi)$ and $\gamma(a),\gamma(b) \in \p \Omega$ is open and dense in $T^*_-\Omega.$
\end{lemma}

\proof Our restriction to foliation upward covectors is needed to avoid total internal reflections, which would prevent $\gamma$ from reaching the boundary.

By symmetry, it suffices to show that we can find $\gamma$ with one endpoint, say $b$, on $\p \Omega$. Let $\gamma: T \to T^* \RR^3$ be the unique maximal purely transmitted bicharacteristic with $\gamma(0) = (x,\xi)$, and let $\gamma$ be its (continuous) projection onto $\RR^3$. If $\gamma(b) \in \p T^*\Omega$ we are done, so assume this does not hold, and let $s = \text{sup } I$.

If $s < \infty$, then $\gamma(s) \in \Gamma$ since otherwise the geodesic could be continued. There are two possibilities: $\gamma$ glances off $\Gamma$ ($\gamma(s) \in T^*\Gamma$), or there is total internal reflection. In the first case, note that $\gamma$ is in the flowout of $T^*\Gamma \setminus 0$ under $\Phi$; this has measure zero in $T^*\Omega$ because $\Phi$ is piecewise smooth and $\text{dim} T^* \Gamma = \text{dim}T^*\Omega -2$.

In the second case, $c$ is smaller on the side of $\Gamma$ opposite $\gamma(s^-)$ by \ref{l: mono increase} convexity of the interfaces, as noted in the proof of Theorem \ref{thm: uniqueness final thm}. This rules out internal reflection, so $\gamma$ can be continued past $s$, a contradiction.

Let us suppose now that $s= \infty$. By lemma \ref{l: mono increase}, $\rho \circ \gamma$ is increasing on $(0,\infty)$. Let $\rho^* = \inf_\gamma \rho$, and choose a sequence $s_j \to \infty$ such that $\rho(\gamma(s_j)) \to \rho^*$. By compactness, $s_j$ has a subsequence (which we may again label $s_j$) such that $\gamma(s_j)$ converges to some point $(x,\xi) \in T^*\Omega$, and by continuity $(x,\xi) \notin T^*_+\Omega$. However, by strict convexity the geodesic starting at any $(x,\xi) \in T^*_+\Omega$ immediately leaves $T^*\Omega \setminus T^*\Omega_\tau$. This is true even if $x \in \Gamma$. By continuity, this is true if we replace $(x,\xi)$ by any sufficiently close covector and in particular $\gamma(s_j)$ for sufficiently large $j$ (as noted above, total internal reflection cannot occur). Hence $\rho^*$ cannot be the infimum of $\rho$ on $\gamma$, a contradiction. $\Box$
\\

The density allows us to just recover the wavespeeds at points where all possible rays through the point never glance.
\\
\\

\emph{Proof of Proposition \ref{prop: internal source construction}}
Because any broken ray intersects only finitely many interfaces in the time interval $t \in [0,2T]$, the condition of being $(\pm)$-escapable is open, and in particular $\mathcal S$ is open.
\subsubsection*{Construction of $h_0$} Let us give a brief argument on one possible construction of the $h_0$ described in the proof just after the statement of Proposition \ref{prop: internal source construction}. First define $J^-\BtoB=vJ\BtoB v$, which is like $J\BtoB$ but propagating backwards in time. Let $v$ be distribution with wavefront set $\mathbb{R}(x, \xi)$ and let $d$ be the number of interfaces between $x$ and $\Omega^c$. Then define $h_0= J\CtoB^{-1}M_T^{-1}\Pi_{i=1}^d(J\BtoB^{-}M_T^{-1})v|_{\Gamma}$. The wavefront set of $h_0$ (viewed in the cosphere bundle) will consists of up to $2^d$ covectors (see Figure \ref{f: subsurface ray}).

\subsubsection*{Construction of $K_{tail}$}
We first define FIO's $\Xi^I_{\pm},\Xi^O_{\pm}:\mathcal C^\infty(\RR \times \p Z) \to \mathcal D'(\mathbf Z)$ of order $0$ producing tails outside $\Theta$ for $(\pm)$-escapable bicharacteristics. The $\Xi^{I/O}_+$-constructed tail for a singularity on a $(+)$-escapable bicharacteristic ensures this singularity escapes $\Theta$ by time $T+T_s$, without generating any singularities in $h\MDT$'s $\PS$-domain of influence where $h\MDT$ is associated to a purely transmitted $\PS$-ray. The $\Xi^{I/O}_-$-constructed tail generates a given singularity on a $(-)$-escapable bicharacteristic, again without causing any singularities in the $\PS$ domain of influence.
The $\Xi^O_{\pm}$ are defined on outgoing boundary data while the $\Xi^I_{\pm}$ are defined on incoming data, microlocally near the final, resp., initial covectors of $(\pm)$-escapable bicharacteristics.

Let $\gamma: (t_-,t_+) \to T^*\mathbf Z$ be $(\pm)$-escapable bicharacteristic. Denote by $\beta^O$ the pullback to the boundary of its finals point: $\beta^O = (d\iota_\Gamma)^*\gamma(t_\pm)$, where by abuse of notation we consider $\gamma(t_\pm)$ as a space-time covector $T^*(\RR \times \mathbf Z)$. Define $\beta^I=(d\iota_\Gamma)^*\gamma(t_\mp)$ similarly. We now define $\Xi_\pm^{I/O}$ microlocally near $\beta^{I/O}$, starting with the incoming maps $\Xi_{\pm}^I$.

\begin{enumerate}
\item[$\bullet$] If $t_\pm \in (0,T+T_s)$: We simply follow the bicharacteristic and apply $\Xi^O_\pm$ at the other end. In the $(+)$ case define $\Xi^I_+\equiv \Xi_+^OJ\BtoB$. In the $(-)$ case, define $\Xi^I_- \equiv \Xi_-^OJ^-\BtoB$ near $\beta^I$, where $J^-\BtoB=vJ\BtoB v$ is like $J\BtoB$ but propagating backwards in time.

\item[$\bullet$] If $\gamma$ escapes, $t_\pm \notin [0,T+T_s]$: This is the terminal case. In the $(+)$ case, there is nothing to do: define $\Xi_+ \equiv 0$ near $\beta^I$. For the $(-)$ case, define $\Xi_-^I \equiv J^{-1}\CtoB$ near $\beta^I$ to obtain the necessary Cauchy data,
\end{enumerate}

We now turn to $\Xi^O_\pm$, considering each of the cases of in the definition of $(\pm)$-escapability.

\begin{enumerate}
\item[$\bullet$] if $\gamma$ escapes: This case never arises: $\Xi^I_\pm$ is not defined in terms of $\Xi^O_\pm$ for such $\gamma.$

\item[$\bullet$] If all outgoing bicharacteristics are $(\pm)$-escapable: Recursively apply $\Xi^I_\pm$ to the reflected and transmitted (if any) bicharacteristics, defining $\Xi^O_\pm \equiv \Xi^I_\pm M$ near $\beta^O$.
\item[$\bullet$] If all the reflecting bicharacteristics are $(\pm)$-escapable, and the opposite incoming $\PS$ rays are $(\mp)$-escapable:
    This is the core case. In the $(+)$ case, near $\beta^O$ let
    \[
    \Xi^O_+ \equiv
-\Xi^I_-M_T^{-1}M_R + \Xi^I_+(M_T-M_RM_T^{-1}M_R),
    \]
The inverses here are all microlocal near the appropriate covector. The $(-)$ case is slightly different: near $\beta^O$,
 \[
    \Xi^O_- \equiv
-\Xi^I_-M_T^{-1} + \Xi^I_+M_RM_T^{-1}.
    \]

\end{enumerate}
Given $\eta \in \mathcal S \subset T^*\Theta^o$, consider all the \bad \ reflecting, $+$-escapable rays associated to $\eta$. Each is associated with a distinct sequence of reflections and transmissions $a=(a_1,\dots,a_k) \in \{R,T\}^k$ for some $k$ and corresponding $\PS$ wave microlocal mode projections  $\Pi_{\lambda_j}$, $\lambda = (\lambda_1,\dots,\lambda_k)$, and a corresponding propagation operator
\[
\mathcal P_{a,\lambda,R} = J\BtoB \Pi_{\lambda_k}M_{a_k} \cdots J\BtoB \Pi_{\lambda_2}M_{a_2}J\BtoB \Pi_{\lambda_1}M_{a_1}J\CtoB.
\]
Notice the $\lambda$ is here so that we are observing the wave (with possible reflection, transmission, and mode conversions) associated to a single broken bicharacteristic consisting of a concatenation of $P$ and $S$ rays. Likewise we can define $\mathcal P_{a,\lambda,T}$ for the transmitting, bad rays that are minus escapable. These transmitting bad rays are new for the systems setting due to multiple wave speeds and were not present in the acoustic setting of \cite{CHKUControl}.

First define
\[
A_\eta = \Xi_+^O\sum_{(a, \lambda) \in \mathfrak G} \mathcal P_{a, \lambda, R}
+\Xi^O_- \sum_{(a,\lambda)\in \mathfrak G^-} \mathcal P_{a,\lambda,T},
\]
and then define $A$ by patching together the $A_\eta$ with a microlocal partition of unity as in \cite{CHKUControl}.
 Given an $h_0$, the tail is precisely
\[
K\tail:= Ah_0
\]

The rest of the proof follows simply by construction of $\Xi_+^O$ and $\Xi_-^I$.
Recall our construction that inside $T^*\Omega_\tau$, $\WF(R_T h_0) = \WF(v)$ for some large enough $T$. One just needs $T$ to be greater than the $P$ or $S$ (depends on which case in the proposition we are considering) distance between $(x,\xi)$ and $S^*\Omega^c$, and one can increase $T$ after that by adjusting $h_0$.
 Set $h_\infty = h_0 + K\tail$, and we must verify that $Fh_\infty|_{\Omega_\tau} \equiv v$ for $t \geq T_s$. Any other waves in this region may only come from $\mathcal P_{a,\lambda,T}h_0$ or $R_t\Xi^O_{\pm} \mathcal P_{a,\lambda, R/T}h_0$ for some $t$ and $(a,\lambda) \in \mathfrak G^\pm$. But by construction of $\Xi^O_{\pm}$, any such bad wave from $\mathcal P_{a,\lambda, T}h_0$ will get cancelled by $\Xi^O_\pm\mathcal P_{a,\lambda, R/T}h_0$. The recursive definition also ensures that any new bad wave created by $\Xi^O_\pm\mathcal P^{(l)}_{a,\lambda}h_0$ also gets eliminated. Thus, neither of these constituents may produce waves whose singularities enter $\Omega_\tau$ microlocally and that completes the proof. $\Box$

\subsubsection*{Funding Acknowledgements} P.C. and V.K. were supported by the Simons Foundation under the MATH + X program. M.V.dH. was partially supported by the Simons Foundation under the MATH + X program, the National Science Foundation under grant DMS-1815143, and by the members of the Geo-Mathematical Group at Rice University. G.U. was partially supported by NSF, a Walker Professorship at UW and a Si-Yuan Professorship at IAS, HKUST.

\newpage
\bibliographystyle{plain}
\bibliography{ScatteringControl}

\end{document}